\documentclass[11pt,a4paper]{article}
\usepackage[T1]{fontenc}
\usepackage[utf8]{inputenc}
\usepackage{authblk}
\usepackage{amssymb, amsmath}
\usepackage{amsthm, amsfonts,mathrsfs}
\usepackage{times}
\usepackage{graphicx}
\usepackage{subfigure}
\usepackage{cite}
\usepackage{hyperref}
\usepackage{epsfig}
\usepackage{color}
\usepackage[all]{xy}
\usepackage{microtype}
\usepackage{setspace}
\usepackage{sectsty}
\subsectionfont{\normalfont\mdseries}
\allowdisplaybreaks
\hypersetup{colorlinks=true,citecolor=blue,linkcolor=blue}
\newtheorem{thm}{Theorem}[section]

\newtheorem{definition}[thm]{Definition}
\newtheorem{lemma}[thm]{Lemma}
\newtheorem{prop}[thm]{Proposition}
\newtheorem{remark}[thm]{Remark}

\numberwithin{equation}{section}

\begin{document}
\title{Smoluchowski-Kramers Approximation for Stochastic Differential Equations driven by Fractional Brownian Motion}

	\author{
	Jiaxin Zha$^1$\thanks{Corresponding author, Email: 18957844432@163.com}\\
		{\textsuperscript{1}\footnotesize\itshape School of Mathematics, Nanjing University of Aeronautics and Astronautics}\\
		{\footnotesize\itshape Nanjing, Jiangsu 211106, P.R. China}
	}
	\date{\vspace{-40pt}}
	\maketitle

	\maketitle
	\noindent{\small{\hspace{1.1cm} }}
	
	\noindent \textbf{Abstract~}   In this paper, we discuss the validity of an approximation inspired by the Smoluchowski–Kramers approximation for a class of stochastic differential equations driven by fractional Brownian motion with additive noise. By rewriting such equations in the form of slow–fast systems and decomposing the fast component into three parts, we investigate the small mass limit of these equations and derive the corresponding convergence rates. Furthermore, under certain regularity conditions, we study the large and moderate deviation principles for a class of stochastic differential equations driven by fractional Brownian motion with small multiplicative noise via the weak convergence approach.  
	\\[2mm]
	\\[2mm]
	\section{Introduction}
	\subsection{Smoluchowski-Kramers approximation}
	
	According to Newton's law, the motion of a particle with mass $\mu$ $(0 < \mu \ll 1)$ in a force field \(b(X)+\sigma(X) \dot{B}\), where the friction force is proportional to the velocity, can be described by
	\begin{equation*}
          \mu \ddot{X}^{\mu}(t) = b(X^{\mu}(t)) + \sigma(X^{\mu}(t)) \dot{B}(t) - \alpha \dot{X}^{\mu}(t), 
          \quad X^{\mu}(0) = x_0, \quad \dot{X}^{\mu}(0) = y_0,
        \end{equation*}
	where \(X^\mu(t)\) is the position of the particle at time \(t\), \(\dot{X}^{\mu}(t)\) can be regarded as the velocity, \(b(X)\) is the deterministic part of the force, \(\sigma(X)\) is the intensity of the noise, \(\sigma(X) \dot{B}\) is the random part of the force, and \(b: \mathbb{R}^{d} \to \mathbb{R}^{d}\) and \(\sigma: \mathbb{R}^{d} \to \mathbb{R}^{d} \times \mathbb{R}^{d}\). The term \(\dot{B}(t)\) is standard \(\mathbb{R}^{d}\)-valued Gaussian white noise, and the term \(\alpha \dot{X}^{\mu}(t)\) describes the resistance to motion (friction), where the friction coefficient \(\alpha\) is a fixed positive constant. Without loss of generality, we can set \(\alpha=1\). When \(\mu\rightarrow 0\), \(X^\mu(t)\) can be approximated by the solution of the first-order equation
	\begin{equation*}
            \dot{X}(t) = b(X(t)) + \sigma(X(t)) \dot{B}(t),\quad X(0) = x_0.
        \end{equation*}
	Formally, the limiting equation is obtained by neglecting the term \(\mu \ddot{X}^{\mu}(t)\) (see in \cite{smoluchowski1916drei} and \cite{kramers1940brownian}).
	
	For each \(T>0\), \(\eta>0\),
	\begin{equation*}
		\lim _{\mu \to 0} \mathbb{P}\left( \max _{0 \leq t \leq T} \left| X^{\mu}(t)-X(t) \right| > \eta \right) = 0,
	\end{equation*}
	this statement is called Smoluchowski–Kramers approximation of $X^{\mu}(t)$ by $X(t)$ (see in Freidlin \cite{freidlin2004some}).
	This result justifies the use of a first-order equation instead of a second-order equation to describe the motion of small particles perturbed by Gaussian white noise. Moreover, it is more convenient to use the solution \(X(t)\) of the first-order stochastic differential equation. Due to its applications, the Smoluchowski–Kramers approximation has been deeply studied by many scholars. Among them, Freidlin \cite{freidlin2004some}, Cerrai and Freidlin \cite{cerrai2006smoluchowski}, Cerrai and Salins \cite{cerrai2014smoluchowski}, Cerrai and Freidlin \cite{cerrai2017smoluchowski}, and He et al. \cite{he2019parameter} have established various Smoluchowski-Kramers approximation results for stochastic equations driven by Gaussian white noise, covering both finite-dimensional and infinite-dimensional cases. These studies fully exploit the Markov property and semimartingale nature of Gaussian white noise, thereby laying a solid foundation for research in this field.
    
    However, in practical systems such as physics, biology, and finance, random noise often exhibits more complex statistical properties than Gaussian white noise, for instance, long-range dependence or significant short-term autocorrelation. This necessitates the study of scenarios driven by fractional Brownian motion. Compared with classical Brownian motion (\(H = \frac{1}{2}\)), the properties of fractional Brownian motion are determined by its Hurst parameter \(H \in (0,1)\), which exhibit distinct mathematical characteristics: when \(H > \frac{1}{2}\), its increments display long-term positive correlation (persistence); when \(H < \frac{1}{2}\), they show negative correlation (anti-persistence). More importantly, for all cases where \(H \neq \frac{1}{2}\), fractional Brownian motion is non-Markovian and  not a semimartingale,  rendering classical analytical tools based on Itô stochastic integration theory inapplicable.
    Boufoussi and Tudor \cite{boufoussi2005kramers} studied the Smoluchowski–Kramers approximation for stochastic equations driven by fractional Brownian motion. However, their analysis was confined to qualitative convergence for models with additive noise and did not provide a convergence rate. The aim of the present paper is to investigate a more general model with power-law intensity \(\mu^\alpha\) noise, prove the validity of the corresponding Smoluchowski-Kramers approximation, and provide the explicit convergence rate for \(0 < \alpha < 1\). 
    
	In the first part of this paper, we consider the following stochastic differential equation driven by fractional Brownian motion with additive noise
	\begin{equation}
		\begin{cases}
			\mu \ddot{X}^{\mu}(t) + \dot{X}^{\mu}(t) = b(X^\mu(t)) + \mu^\alpha \dot{B}^{H}(t), \\
			X^\mu(0) = x_0, \quad \dot{X}^{\mu}(0) = y_0.
		\end{cases}\label{wwave}
	\end{equation}
	Here, \(0 \leq \alpha < 1\) is a constant, and \({B^H(t), t \in[0, T]}\) is a fractional Brownian motion with Hurst parameter \(H \in (\frac{1}{2}, 1)\). System (\ref{wwave}) describes the motion of a particle with mass $\mu$ in an environment without thermal fluctuations, and can be regarded as a singularly perturbed differential equation system with stochastic noise, where the noise intensity is deterministically controlled by \( \mu^\alpha\).
	Formally, the effective approximation model for (\ref{wwave}) can also be obtained by omitting the term \(\mu \ddot{X}^{\mu}(t)\), i.e.,
	\begin{equation}
		\begin{cases}
			\dot{\bar{X}}^\mu(t) = b({\bar{X}}^\mu(t)) + \mu^\alpha \dot{B}^{H}(t), \\
			\bar{X}^\mu(0) = x_0.\label{aayeq}
		\end{cases}
	\end{equation}
	Clearly, when \(\alpha=0\), the above statement reduces to the classical Smoluchowski–Kramers approximation; in this case, (\ref{aayeq}) does not depend on \(\mu\), and we can use \(\bar{X}\) instead of \(\bar{X}^{\mu}\). This paper systematically investigates the Smoluchowski--Kramers approximation for stochastic systems driven by fractional Brownian motion by rewriting the original second-order system as a slow--fast system and employing a solution splitting technique (see in \cite{wang2011approximation} or \eqref{decomposition} in this paper) to split the fast variable into three independently analyzable parts. Under suitable assumptions, we rigorously prove that as the mass parameter \(\mu \to 0\), the position process \(X^{\mu}(t)\) of the original system converges to the process \(\bar{X}^{\mu}(t)\) described by a first-order approximate system. Furthermore, we precisely characterize the convergence behavior under different noise intensities: when the noise intensity is \(\mu^{\alpha}\) with \(0 < \alpha < 1\), an explicit convergence rate of order \(\mu^{\alpha}\) is obtained; whereas in the classical case \(\alpha = 0\), while convergence in the mean sense is established, no explicit convergence rate can be derived.
    
    To provide a deeper characterization of the fine asymptotic behavior of the system when \(\alpha=0\), we further introduce a small perturbation parameter \(\varepsilon^H\) in the second part to study the probabilistic asymptotic properties as \(\varepsilon \to 0\). By strengthening the regularity conditions on the noise term \(\sigma\), we successfully extend the model from additive noise to the more general case of multiplicative noise, and establish both the large and moderate deviation principles for the system.

	\subsection{Large Deviation Principle for Stochastic Differential Equations with Small Multiplicative Noise}

    As previously discussed, in the classical case of \(\alpha = 0\), although the system converges in the mean sense, an explicit convergence rate cannot be derived. This highlights the significant limitations of traditional convergence analysis in characterizing fine asymptotic behavior. To address this shortcoming and deeply investigate the mechanism of rare events under small noise perturbations when \(\alpha = 0\), we introduce in this study the powerful probabilistic tools of large and moderate deviation theory. It is worth emphasizing that the study of asymptotic behavior under small noise perturbations extends beyond the convergence of trajectories themselves. In complex systems across physics, biology, finance, and many other fields, noise often induces rare events. For example, the "tunneling" phenomenon, where the system transitions between different stable states, occurs with a probability that tends to zero as the noise intensity decreases. Large deviation theory precisely quantifies the exponential decay rate of such probabilities (see in \cite{freidlin2012random}), while moderate deviation theory further fills the theoretical gap between the central limit theorem and large deviation principles, describing asymptotic behaviors that lie between typical fluctuations and large deviations.
    
    For stochastic systems driven by classical Brownian motion (Gaussian white noise), the large deviation theory is well-established. Early work is represented by the framework developed by Freidlin and Wentzell \cite{freidlin2012random} for classical diffusion processes. With the development of stochastic models, researchers began to focus on more complex noise structures. Among these, systems driven by fractional Brownian motion present fundamentally new challenges for large deviation analysis due to their long- or short-range dependence and non-Markovian properties. In the study of large deviations for systems driven by fractional Brownian motion, Li and Qian \cite{li2021large} established a capacity-based large deviation principle for fractional Brownian motion with Hurst parameter \(H \geq \frac{1}{2}\) on the classical Wiener space. However, their work primarily focused on additive noise scenarios and qualitative analysis of the process itself, without extending it to stochastic differential equation systems. Budhiraja and Song \cite{budhiraja2025large} employed weak convergence and variational methods to construct large deviation principles for functionals and stochastic dynamical systems driven by fractional Brownian motion with \(H > \frac{1}{2}\). Nevertheless, their analysis was confined to first-order systems and did not investigate deviation behaviors under the physical approximation where the mass also tends to zero. 
    
    In the field of moderate deviations, Bourguin et al. \cite{bourguin2024moderate} studied the moderate deviation principle for slow-fast systems driven by fractional Brownian motion with \(H \in (\frac{1}{2},1)\), revealing a discontinuity in the action functional at \(H=\frac{1}{2}\), but did not deeply explore the transitional relationship between the moderate deviation scale, the central limit theorem, and large deviations. Subsequently, Yang et al. \cite{yang2024moderate} combined weak convergence methods with a Khasminskii-type averaging principle to establish a moderate deviation principle for two-time-scale systems involving mixed fractional Brownian motion, thereby filling the gap between the central limit theorem and large deviation asymptotics for multiscale systems. However, their analysis primarily relied on the \(\alpha\)-Hölder space topology and required the fast variable to depend entirely on the slow variable, which limits the applicability of the results to practical problems.
	
	Therefore, the second goal of this paper is to establish large and moderate deviation principles for the stochastic differential equation with small noise driven by fractional Brownian motion
	\begin{equation}
		\begin{cases}
			\mu \ddot{X}^{\mu,\varepsilon}(t) + \dot{X}^{\mu,\varepsilon}(t) = b({X}^{\mu,\varepsilon}(t)) + {\varepsilon}^{H}\sigma({X}^{\mu,\varepsilon}(t))\dot{B}^H(t),\\
			{X}^{\mu,\varepsilon}(0) = x_0,\quad \dot{X}^{\mu,\varepsilon}(0) = y_0,
		\end{cases}\label{original1}
	\end{equation}
	where $\mu = \mu(\varepsilon)$ satisfies $\mu \to 0$ as $\varepsilon \to 0$. Theorem \ref{ldpresult} proves that the process ${X}^{\mu,\varepsilon}(t)$ satisfies a large deviation principle with speed $\varepsilon^{2H}$ and rate function $I$. Intuitively, this result states that for any Borel set $A \subset C([0,T];\mathbb{R}^{d})$, we have \(\mathbb{P}\{{X}^{\mu,\varepsilon}(t) \in A\} \approx \exp\left\{ -\varepsilon^{-2H} \inf_{g \in A} I(g) \right\}.\)
	
	The large deviation principle mainly characterizes the asymptotic behavior of small probability events where ${X}^{\mu,\varepsilon}(t)$ deviates from its average path (with deviation of order $O(1)$). To further study fluctuations of a smaller order of magnitude than those captured by large deviations, we introduce the moderate deviation principle. In Theorem \ref{mdpresult}, we will prove that when considering deviations of order $\varepsilon^H \lambda(\varepsilon)$, the normalized process \(\frac{{X}^{\mu,\varepsilon}(t) - \bar{X}^0(t)}{\varepsilon^H \lambda(\varepsilon)}\)
	satisfies a moderate deviation principle with speed $\lambda^{-2}(\varepsilon)$ and rate function $\bar{I}$. This means that for any Borel set $A \subset C([0,T];\mathbb{R}^{d})$, we have \(\mathbb{P}\left\{ \frac{{X}^{\mu,\varepsilon}(t) - \bar{X}^0(t)}{\varepsilon^H \lambda(\varepsilon)} \in A \right\} \approx \exp\left\{ -\lambda^{2}(\varepsilon) \inf_{g \in A} \bar{I}(g) \right\},\)
	where $\bar{X}^0(t)$ is the solution of the corresponding "average" equation (\ref{sim}).
	
	Regarding the proof method, we adopt the weak convergence approach to establish the large and moderate deviation principles for the family of stochastic processes $\{{X}^{\mu,\varepsilon}\}_{\varepsilon > 0}$ defined by (\ref{original1}). This method was systematically applied to the study of large deviation problems by Budhiraja and Dupuis \cite{budhiraja2019analysis}, Dupuis and Ellis\cite{dupuis2011weak}. Its core lies in the equivalence between the large deviation principle and the Laplace principle established by Dupuis and Ellis \cite{dupuis2011weak}, and the variational representations of exponential functionals of Wiener processes developed by Budhiraja and Dupuis \cite{budhiraja2000variational}. A set of sufficient conditions ensuring the validity of the Laplace principle was provided by Budhiraja, Dupuis and Maroulas \cite{budhiraja2008large}, building on the foundational work of Budhiraja and Dupuis \cite{budhiraja2000variational}, while Matoussi, Sabbagh and Zhang \cite{matoussi2021large} further extended these related results. Under this framework, we will use Lemma \ref{LDP} ( given in Matoussi, Sabbagh and Zhang \cite{matoussi2021large}) as the main tool to facilitate the derivation of the large deviation principle.  We also referred to relevant papers such as \cite{fan2023}, \cite{shen2024} and \cite{li2025} for the study of large deviations in stochastic differential equation driven by Gaussian process.
    
	The structure of the subsequent parts of this paper is as follows. In Section \ref{Pre}, we review some concepts of stochastic calculus for fractional Brownian motion, introduce the conditions on the coefficients, and recall the large deviation principle. In Section \ref{result}, we prove the validity of a Smoluchowski-Kramers type approximation for a class of stochastic differential equations driven by fractional Brownian motion, and provide the corresponding moment estimates along with detailed proofs. Section \ref{sec large} is devoted to establishing the large deviation principle for the small-noise stochastic differential equation driven by fractional Brownian motion. Specifically, we derive the large deviation principle for the family of processes $\{{X}^{\mu,\varepsilon}\}_{\varepsilon>0}$ under the condition that $\mu(\varepsilon)\to 0$ as $\varepsilon \to 0$. In Section \ref{sec mod}, these results are extended to the moderate deviation regime, where we establish the corresponding moderate deviation principle for the same family.
	
	We note that throughout the paper, $C$ denote positive constants that may change from line to line, and the value of $C$ may depend on other constants such as $T$ and the Hurst parameter $H$ , their values will be indicated in the subscript of $C$, but it never depends on $\mu$. We use the notation $x\lesssim y$ to indicate that there exists a constant $C$ such that $x\leq Cy$.
	
	\section{Preliminaries}\label{Pre}
	
	Let $(\Omega,\mathscr{F},\mathbb{P})$ be a complete probability space, on which there is a filtration $(\mathscr{F}_t)_{0\leq t\leq T}$ satisfying the usual condition, where $0<T<\infty$ is fixed throughout the paper.  In the rest,$\langle x, y \rangle := \sum\limits_{i=1}^{d} x_i y_i$ for all $x, y \in \mathbb{R}^d$, $|x|:=\sqrt{\sum\limits_{i=1}^{d}x_i^2}$ for each $x=(x_1,...,x_d)\in\mathbb{R}^d,$ and $||A||:=\sup\limits_{x\in\mathbb{R}^d,|x|=1}|Ax|$ for each matrix $A\in\mathbb{R}^{d\times d}.$ We use $\mathcal{L}^{p}([0,T];\mathbb{R}^{d\times m})$ to denote the family of $\mathbb{R}^{d\times m}$-valued $\mathscr{F}_{t}$-adapted processes $\{f(t)\}_{0\leq t\leq T}$ such that $\int_{0}^{T}|f(t)|^{p}dt<\infty$ a.s., and $\mathcal{M}^{p}([0,T];\mathbb{R}^{d\times m})$ to denote the family of processes $\{f(t)\}_{0\leq t\leq T}$ in $\mathcal{L}^{p}([0,T];\mathbb{R}^{d\times m})$ such that $\mathbb{E}\int_{0}^{T}|f(t)|^{p}dt<\infty$. Let $B^H$ be a $d$-dimensional fractional Brownian motion with Hurst parameter $H \in (\frac{1}{2},1)$ on $(\Omega,\mathscr{F},(\mathscr{F}_t)_{0 \leq t \leq T},\mathbb{P})$.  
	
	Consider the reproducing kernel Hilbert space $\mathbb{H}$:
	\[
	\mathbb{H} = \left\{ h \in C([0,T];\mathbb{R}^d) : h(0) = \mathbf{0},\ \|h\|_{\mathbb{H}} < \infty \right\},
	\]
	where $\mathbf{0}$ denotes the zero vector in $\mathbb{R}^d$, and the norm $\|h\|_{\mathbb{H}}$ is given by
	\[
	\|h\|_{\mathbb{H}}^2 = H(2H-1) \int_0^T \int_0^T h(s)h(t)|s-t|^{2H-2}dsdt.
	\]
	
	There exists a linear isometry between $\mathbb{H}$ and $L^2([0,T];\mathbb{R}^d)$. Specifically, for any $h \in \mathbb{H}$, there exists a unique $u \in L^2([0,T];\mathbb{R}^d)$ such that
	\[
	\int_0^th(s)dB^H(s)=\int_0^t(K_H^*h)(s)d B^{1/2}(s), \quad t \in [0,T],
	\]
    where $B^H$ is a fractional Brownian motion with parameter $H>\frac{1}{2}$, $K_H$ is the operator induced by the square-integrable kernel $K_H(t,s)=c_H (t-s)_+^{H-1/2}\int_0^1u^{H-3/2}(1-(1-t/s)u)^{H-1/2}du$, and  $K_H^*$ is the adjoint of $K_H$ with $(K_H^*h)(s)=\int_s^th(u)\frac{\partial K_H}{\partial u}(u,s)du$.
Then the norms are related by
	\[
	\|h\|_{\mathbb{H}}^2  = \int_0^{\cdot} |(K_H^*h)(s)|^2 ds = \|(K_H^*h)\|_{L^2}^2.
	\]

	Let $\mathscr{A}$ denote the class of $\mathbb{R}^d$-valued $\mathcal{F}_t$-predictable processes $h(\omega, \cdot)$ belonging to $\mathbb{H}$ a.s. For each $N > 0$, let
	\[
	S_N := \left\{ h \in \mathbb{H}; \|h\|_{\mathbb{H}}^2 = \int_0^T |(K_H^*h)(s)|^2 ds \le N \right\}.
	\]
	$S_N$ is endowed with the weak topology induced from $\mathbb{H}$. Define
	\[
	\mathscr{A}_N := \left\{ h \in \mathscr{A}, h(\omega, \cdot) \in S_N, \mathbb{P}\text{-a.s.} \right\}.
	\]
	
	We make the following assumptions about the coefficients. The functions $b(x):\mathbb{R}^{d}\to\mathbb{R}^{d}$, $\sigma(x):\mathbb{R}^{d}\to\mathbb{R}^{d}\times\mathbb{R}^{d}$, satisfy
	
	\textbf{(A1)}. There exists a constant $L>0$, such that for all $x_1,x_2\in\mathbb{R}^d$,
	\[
	|b(x_1)-b(x_2)| + \|\sigma(x_1)-\sigma(x_2)\| \leq L|x_1 - x_2| .
	\]
	
	\textbf{(A2)}. The functions $b,\sigma$ are uniformly bounded and the diffusion matrix $\sigma\sigma^T$ is uniformly nondegenerate.
	
	\textbf{(A3)}. There exists a constant $K>0$, such that for all $x\in\mathbb{R}^d$,
	\[
	\|\sigma(x)\|+\|D_s\sigma(x)\| \leq K,
	\]
	where $D_s$ is the Malliavin derivative with respect to random variables (for the details, we can see Alòs and Nualart \cite{alos2003stochastic}). The condition is similar in form to the assumption used in Shen and Wang \cite{shen2025conditional}, representing a relatively strong regularity requirement. In Shen and Wang \cite{shen2025conditional}, such a condition was successfully applied to address stochastic analysis problems under multiplicative noise, and the system studied in this paper also contains multiplicative stochastic perturbations. Within the framework of Malliavin analysis, adopting a similar regularity condition by referencing their approach is a natural choice, as it allows us to leverage existing estimation techniques and theoretical conclusions. Therefore, adopting this assumption in the present work is both reasonable and standard.
	
	In preparation for subsequently establishing the large deviation principle, we recall the concept of action functional given by Freidlin and Wentzell \cite{freidlin2012random} and the Laplace principle formulated by Dupuis and Ellis \cite{dupuis2011weak}.
	
	\begin{definition}\label{actfunction}
		Let $\mathscr{X}$ be a metric space with metric $\rho$. On the $\sigma$-algebra of its Borel subsets, let $\mu^{\varepsilon}$ be a family of probability measures depending on a parameter $\varepsilon>0$. Let $\lambda(\varepsilon)$ be a positive real-valued function going to $+\infty$ as $\varepsilon\downarrow 0$ and let $S(x)$ be a function on $\mathscr{X}$ assuming values in $[0,\infty]$. We shall say that $\lambda(\varepsilon)S(x)$ is an action function for $\mu^{\varepsilon}$ as $\varepsilon\downarrow 0$ if the following assertions hold:
		
		\begin{itemize}
			\item[(0)] the set $\Phi(s)=\{x:S(x)\leq s\}$ is compact for every $s\geq 0$;
			
			\item[(I)] for any $\delta>0$, any $\gamma>0$ and any $x\in\mathscr{X}$ there exists an $\varepsilon_{0}>0$ such that
			\begin{equation}\label{mu1}
				\mu^{\varepsilon}\{y:\rho(x,y)<\delta\}\geq \exp\{-\lambda(\varepsilon)[S(x)+\gamma]\},
			\end{equation}
			for all $\varepsilon\leq \varepsilon_{0}$;
			
			\item[(II)] for any $\delta>0$, any $\gamma>0$ and any $s>0$ there exists an $\varepsilon_{0}>0$ such that
			\begin{equation}\label{mu2}
				\mu^{\varepsilon}\{y:\rho(y,\Phi(s))\geq \delta\}\leq \exp\{-\lambda(\varepsilon)(s-\gamma)\},
			\end{equation}
			for all $\varepsilon\leq \varepsilon_{0}$.
		\end{itemize}
		
		If $X^{\varepsilon}$ is a family of random elements of $\mathscr{X}$ defined on the probability spaces $(\Omega^{\varepsilon},\mathscr{F}^{\varepsilon},\mathbb{P}^{\varepsilon})$, then the action function for the family of the distributions $\mu^{\varepsilon}$, $\mu^{\varepsilon}(A)=\mathbb{P}^{\varepsilon}(X^{\varepsilon}\in A)$ is called the action function for the family $X^{\varepsilon}$.In this case formulas (\ref{mu1}) and (\ref{mu2}) take the form:
		\[
		\begin{aligned}
			&\mathbb{P}^\varepsilon\{\rho(X^\varepsilon, x) < \delta\} \geq \exp\{-\lambda(\varepsilon)[S(x) + \gamma]\}, \\
			&\mathbb{P}^\varepsilon\{\rho(X^\varepsilon, \Phi(s)) \geq \delta\} \leq \exp\{-\lambda(\varepsilon)(s-\gamma)\}.
		\end{aligned}
		\]
		
		Separately, the functions $S(x)$ and $\lambda(\varepsilon)$ will be called the normalized action function and normalizing coefficient.
	\end{definition}
	\begin{definition}
		Let $\{X^\varepsilon,\varepsilon>0\}$ be a family of random variables taking values in a Polish space $\mathscr{X}$. Let $\eta(\varepsilon)$ be a positive real-valued function going to $+\infty$ as $\varepsilon\downarrow 0$ and let $I$ be a rate function on $\mathscr{X}$. We say that $\{X^\varepsilon,\varepsilon>0\}$ satisfies the Laplace principle with speed $\eta^{-1}(\varepsilon)$ and rate function $I$ if for every bounded and continuous function $f:\mathscr{X}\to\mathbb{R}$
		\[
		\lim_{\varepsilon\to 0} -\eta^{-1}(\varepsilon)\ln\mathbb{E}\left[\exp\left\{-\eta(\varepsilon)f(X^\varepsilon)\right\}\right] = \inf_{x\in\mathscr{X}}[I(x) + f(x)].
		\]
	\end{definition}
	
	If the rate function has compact level sets, then the Laplace principle is equivalent to the corresponding large deviations principle in Definition \ref{actfunction} with the same rate function ($S(x)=I(x)$, 
    $\lambda(\varepsilon)=\eta(\varepsilon)$). For a detailed proof, we refer the reader to Section 1.2 of Dupuis and Ellis \cite{dupuis2011weak}.
	We next summarize a set of sufficient conditions for the Laplace principle to hold, as established by Budhiraja and Dupuis \cite{budhiraja2000variational} and Budhiraja, Dupuis and Maroulas \cite{budhiraja2008large}.
	
	\begin{lemma}\label{ldp}\cite[Theorem 5]{budhiraja2008large}
		For any $\varepsilon>0$, let $\Gamma^\varepsilon$ be a measurable mapping from $C([0,T];\mathbb{R}^d)$ into $C([0,T];\mathbb{R}^d)$. Suppose that $\{\Gamma^\varepsilon\}_{\varepsilon>0}$ satisfies the following assumptions: there exists a measurable map $\Gamma^0:C([0,T];\mathbb{R}^d)\to C([0,T];\mathbb{R}^d)$ such that
		
		\begin{itemize}
			\item[(a)] For every $N<\infty$, the set $\{\Gamma^0\left(\int_0^{\cdot} h(s) ds\right) ; h \in S_N\}$ is a compact subset of $C([0,T];\mathbb{R}^d)$;
			
			\item[(b)] Consider $N<\infty$ and a family $\{h^\varepsilon\}_{\varepsilon>0}\subset\mathscr{A}_N$ such that $h^\varepsilon$ converges in distribution as $S_N$-valued random variables to $h$ as $\varepsilon\to 0$, then $\Gamma^\varepsilon\left({\varepsilon}^{H}B^H({\cdot})+ \int_0^{\cdot} h^\varepsilon(s) ds\right)$ converges in distribution to $\Gamma^0\left(\int_0^{\cdot} h(s) ds\right)$.
		\end{itemize}
		
		Let $X^\varepsilon = \Gamma^\varepsilon({\varepsilon}^{H}B^H(t))$, then the family $\{X^\varepsilon\}_{\varepsilon>0}$ satisfies a large deviation principle in $C([0,T];\mathbb{R}^d)$ with speed $\varepsilon^{2H}$ and rate function $I$ given by
		\[
		I(g) = \inf_{h \in \mathbb{H},\; g = \Gamma^0\left(\int_0^{\cdot} h(s) ds\right)} \left\{ \frac{1}{2} \|h\|_{\mathbb{H}}^2 \right\}, \quad g \in C([0,T];\mathbb{R}^d)
		\]
		with $\inf\emptyset = \infty$ by convention.
	\end{lemma}
	
	The following result was proved by Matoussi, Sabbagh and Zhang \cite{matoussi2021large}, who provided a convenient sufficient condition for verifying the assumptions in Lemma \ref{ldp} that is particularly suitable for our current setting.
	
	\begin{lemma}\label{LDP}\cite[Theorem 3.2]{matoussi2021large}
		For any $\varepsilon>0$, let $\Gamma^\varepsilon$ be a measurable mapping from $C([0,T];\mathbb{R}^d)$ into $C([0,T];\mathbb{R}^d)$. Suppose that $\{\Gamma^\varepsilon\}_{\varepsilon>0}$ satisfies the following assumptions: there exists a measurable map $\Gamma^0:C([0,T];\mathbb{R}^d)\to C([0,T];\mathbb{R}^d)$ such that
		
		\begin{itemize}
			\item[(i)] Let $\{h^\varepsilon\}_{\varepsilon>0}\subset S_N$ for some $N<\infty$ such that $h^\varepsilon$ converges to element $h$ in $S_N$ as $\varepsilon\to 0$, then $\Gamma^0\left(\int_0^{\cdot} h^\varepsilon(s) ds\right)$ converges to $\Gamma^0\left(\int_0^{\cdot} h(s) ds\right)$ in $C([0,T];\mathbb{R}^d)$;
			
			\item[(ii)] Let $\{h^\varepsilon\}_{\varepsilon>0}\subset\mathscr{A}_N$ for some $N<\infty$. For any $\delta>0$, we have
			\[
			\lim_{\varepsilon\to 0} \mathbb{P}\left\{ d\left( \Gamma^\varepsilon\left( {\varepsilon}^{H}B^H({\cdot})+ \int_0^{\cdot} h^\varepsilon(s) ds \right), \Gamma^0\left( \int_0^{\cdot} h^\varepsilon(s) ds \right) \right) > \delta \right\} = 0,
			\]
			where $d(\cdot,\cdot)$ denotes the metric in the space $C([0,T];\mathbb{R}^d)$.
		\end{itemize}
		
		Let $X^\varepsilon = \Gamma^\varepsilon({\varepsilon}^{H}B^H$), then the family $\{X^\varepsilon\}_{\varepsilon>0}$ satisfies a large deviation principle in $C([0,T];\mathbb{R}^d)$ with speed $\varepsilon^{2H}$ and rate function $I$ given by
		\[
		I(g) = \inf_{h \in \mathbb{H},\; g = \Gamma^0\left( \int_0^{\cdot} h(s) ds \right)} \left\{ \frac{1}{2} \|h\|_{\mathbb{H}}^2 \right\}, \quad g \in C([0,T];\mathbb{R}^d)
		\]
		with $\inf\emptyset = \infty$ by convention.
	\end{lemma}

    In the main proof process of this paper, we also need to use the following technical lemmas.
	\begin{lemma}\cite[Theorem1.10.3]{mishura2008stochastic}
		Let \( H \in (\frac{1}{2}, 1) \), \( f \in L_{\frac{1}{H}}[0,T] \), and define
		\[
		I_t(f) = \int_0^t f(s) dB^H(s).
		\]
		Then for any \( p > 0 \), there exists a constant \( C_p(H) > 0 \) such that
		\begin{equation}
			\left\| \sup_{0 \leq t \leq T} |I_t(f)| \right\|_p \leq C_p(H) \|f\|_{L_{\frac{1}{H}}[0,T]},
			\label{moment_inequality_fBm}
		\end{equation}
		where \( \| \cdot \|_p \) denotes the \( L^p \)-norm with respect to the probability measure.
	\end{lemma}
	
	\begin{lemma}\cite[Remark 5]{alos2003stochastic}
		Let \( H \in (\frac{1}{2}, 1) \), and \( u = \{u_t, t \in [0,T]\} \) be a stochastic process in the space \( \mathbb{L}_{H}^{1,p} \), which is defined as the collection of \( \mathcal{F}_t \)-adapted processes such that \( u \in L^p(\Omega \times [0,T]) \) and its Malliavin derivative \( Du \) satisfies 
    \[
    \mathbb{E}\left[ \int_0^T \left( \int_0^T |D_s u_r|^{\frac{1}{H}} ds \right)^{pH} dr \right] < \infty,
    \]
    for \( p > \frac{1}{H} \).  Then
		\begin{equation}
                \mathbb{E}\left[ \sup_{t \in [0,T]} \left| \int_0^t u_s d B^H(s) \right|^p \right] 
                \leq C \left[ \int_0^T |\mathbb{E} u_s|^p ds + \mathbb{E} \int_0^T \left( \int_0^T |D_s u_r|^{\frac{1}{H}} ds \right)^{pH} dr \right],
                \label{maximal_inequality}
            \end{equation}
		where the constant \( C > 0 \) depends on \( p \), \( H \) and \( T \).
	\end{lemma}
	
	\begin{remark}
		Throughout this paper, we adopt the notation $\dot{B}^H(t)$ to denote the formal derivative of fractional Brownian motion $B^H(t)$, even though it is well-known that for $H \neq \frac{1}{2}$, $B^H(t)$ is not a semimartingale and its sample paths are almost surely not differentiable. This notation is used in a formal sense to facilitate the presentation of stochastic differential equations driven by fractional Brownian motion.  This paper considers the case \( H \in (\frac{1}{2}, 1) \), If $f$ is non-random, any integral of the form \(\int_0^t f(s) \dot{B}^H(s)  ds\)  is to be understood pathwise as the Young integral \(\int_0^t f(s)  dB^H(s)\). This integral is well-defined due to the Hölder regularity of the sample paths of $B^H(t)$ and the solution processes considered (see in \cite{mishura2008stochastic}). If $f$ is random, the integral will be considered as a divergence type integral, and its upper bound can be estimated by \eqref{maximal_inequality}.
	\end{remark}
	
	\section{Smoluchowski-Kramers approximation}\label{result}
	In this section, we consider the following stochastic differential equation driven by fractional Brownian motion with additive noise

	\begin{equation}
		\begin{cases} 
			\mu \ddot{X}^{\mu}(t) + \dot{X}^{\mu}(t) = b(X^\mu(t)) + \mu^\alpha \dot{B}^{H}(t) \\ 
			X^\mu(0) = x_0, \quad \dot{X}^{\mu}(0) = y_0. 
		\end{cases}\label{wave}
	\end{equation}

	Formally,  the effective approximation model of (\ref{wave}) can also be obtained by dropping the $\mu \ddot{X}^{\mu}$ term, that is, 
	\begin{equation}
		\begin{cases} 
			 \dot{\bar{X}}^\mu(t) = b({\bar{X}}^\mu(t)) + \mu^\alpha \dot{B}^{H}(t), \\ 
			\bar{X}^\mu(0) = x_0.
		\end{cases}\label{ayeq}
	\end{equation}

	Rewrite the equation (\ref{wave}) as
	\begin{equation}
		\begin{cases}
			\dot{X}^{\mu}(t) = Y^\mu(t), \\ 
			\dot{Y}^{\mu}(t) = \mu^{-1} \left[ -Y^\mu(t) + b(X^\mu(t)) \right] + \mu^{\alpha-1} \dot{B}^{H}(t) ,\\ 
			X^\mu(0) = x_0, \quad Y^\mu(0) = y_0. 
		\end{cases}\label{slowfast}
	\end{equation}
	Equation (\ref{slowfast}) has a form of slow-fast system (see in Duan and Wang \cite{duan2014effective}). Inspired by a splitting technique introduced by Lv et al.\cite{wang2011approximation}, we make the following important decomposition, which makes the analysis to (\ref{slowfast}) considerably more clear	
	\begin{equation}
		\begin{cases}
			\dot{\bar{Y}}^\mu_1(t) = -\mu^{-1} \bar{Y}^\mu_1(t), \\ 
			\dot{\bar{Y}}^\mu_2(t)= -\mu^{-1} [\bar{Y}^\mu_2(t) - b(X^\mu(t))], \\ 
			\dot{\bar{Y}}^\mu_3(t)= -\mu^{-1} \bar{Y}^\mu_3(t) + \mu^{-H} \dot{B}^{H}(t), \\
			\bar{Y}^\mu_1(0) = \mu x_0, \quad \bar{Y}^\mu_2(0) = 0, \quad \bar{Y}^\mu_3(0) = 0. 
		\end{cases}\label{byeq}
	\end{equation}
	Direct calculation yields 
	\begin{equation}
			Y^\mu(t) = \mu^{-1} \bar{Y}^\mu_1(t) + \bar{Y}^\mu_2(t) + \mu^{\alpha+H-1} \bar{Y}^\mu_3(t). \label{decomposition}
	\end{equation}

	Our approximate result is the following theorem.
	
	\begin{thm}\label{mainresult}
		(i) Let $0<\alpha<1.$ Under assumption \textbf{(A1)},
		\begin{equation}
			\mathbb{E}\sup\limits_{0\leq t\leq T}|X^\mu(t)-\bar{X}^\mu(t)|\lesssim \mu^\alpha.
		\end{equation} 	
		(ii) Let $\alpha=0.$ Under assumption \textbf{(A1)},
		\begin{equation}
			\lim_{\mu \to 0} \sup\limits_{0\leq t\leq T} \mathbb{E} |X^\mu(t)-\bar{X}^\mu(t)| = 0.
		\end{equation}
	\end{thm}
	\begin{remark}
		Obviously, part (i) of Theorem \ref{mainresult} does not give a convergence result in the case $\alpha=0$. But we can get an explicit estimate for the convergence rate in the total variation distance between $X^\mu(t)$ and $\bar{X}^\mu(t)$.\cite[Theorem3.1]{son2020rate}
		$$d_{\text{TV}}(X^\mu(t);\bar{X}^\mu(t)) \leq C t^{-H} \mu^H.$$
	\end{remark}
	
	To prove this approximation result, we establish several moment estimates. We start with a well-posedness result. 
	\begin{lemma}\label{lemmaexistence}
		For each $0<\mu\leq 1$, the equation (\ref{slowfast}) admits a unique strong solution.
	\end{lemma}
	\begin{proof}
		Let $\psi^\mu:=\begin{pmatrix}
			X^\mu\\
			Y^\mu
		\end{pmatrix}$. The equation (\ref{slowfast}) can be rewritten as
		\begin{equation*}
                d\psi^\mu(t) = [\mathcal{A}^\mu\psi^\mu(t) + \mathcal{F}^\mu(\psi^\mu(t))]dt + \Sigma^\mu dB^H(t),
            \end{equation*}
		where 
		$\mathcal{A}^\mu := \begin{pmatrix}
			0 & Id\\
			0 & -\mu^{-1}Id
		\end{pmatrix}$, 
		$\mathcal{F}^\mu\begin{pmatrix}
			x\\
			y
		\end{pmatrix} := \begin{pmatrix}
			0\\
			\mu^{-1}b(x)
		\end{pmatrix}$,
		and 
		$\Sigma^\mu := \begin{pmatrix}
			0\\
			\mu^{\alpha-1}Id
		\end{pmatrix}$.
	
		Under assumption \textbf{(A1)}, $b$ is Lipschitz continuous, $\mathcal{A}^\mu+\mathcal{F}^\mu$ is also Lipschitz continuous. The diffusion coefficient $\Sigma^\mu$ is constant. The existence and uniqueness of a pathwise solution for such multidimensional SDEs driven by a fractional Brownian motion with Hurst parameter $H > \frac{1}{2}$ are established by\cite[Theorems 2.1 and 5.1]{rascanu2002differential}. Therefore, the conclusion follows.
	\end{proof}
	In order to give moment estimate for $X^\mu$, we first consider the linear part $x^\mu$ of (\ref{wave}), that is
	\begin{equation}\label{linearwave}
	\mu\ddot{x}^\mu+\dot{x}^\mu=\mu^\alpha\dot{B}^H,\quad x^\mu(0)=0, \dot{x}^\mu(0)=0.
	\end{equation}
	Similar to (\ref{slowfast}), we rewrite it as
	\begin{equation*}
            \begin{cases}
                \dot{x}^\mu=y^\mu,\quad x^\mu(0)=0,\\
                \dot{y}^\mu=-\mu^{-1}y^\mu+\mu^{\alpha-1}\dot{B}^H,\quad y^\mu(0)=0.
            \end{cases}
        \end{equation*}
	
	\begin{lemma}\label{thmforlinearEsup}
		For any $0 \leq t \leq T$ and $0 < \mu \leq 1$, let $x^\mu(t)$ denote the stochastic process defined by equation (\ref{linearwave}). Then the following uniform bound holds
        $$\sup\limits_{0<\mu\leq 1}\mathbb{E}\sup\limits_{0\leq t\leq T}|x^\mu(t)|<\infty.$$
	\end{lemma}
	\begin{proof}
		Set $w^\mu$ to be the solution of the following linear SDE
		\begin{equation}\label{defofw}
			\dot{w}^\mu=-\mu^{-1}w^\mu+\mu^{-H}\dot{B}^H,\quad w^\mu(0)=0,
		\end{equation} 
		 and it is straightforward to check that $y^\mu=\mu^{\alpha+H-1}w^\mu,$ so
		$$x^\mu(t)=\int_{0}^{t}y^\mu(s)ds=\mu^{\alpha+H-1}\int_{0}^{t}w^\mu(s)ds.$$
		By (\ref{defofw}) the definition of $w^\mu$, we have
		$$w^\mu(t)=-\mu^{-1}\int_{0}^{t}w^\mu(s)ds+\mu^{-H}B^H(t).$$
		Multiplying $\mu^{\alpha+H}$ and rearranging,
		$$\mu^{\alpha+H-1}\int_{0}^{t}w^\mu(s)ds=\mu^{\alpha}(B^H(t)-\mu^{H}w^\mu(t)),$$
		so that 
		$$\mathbb{E}\sup\limits_{0\leq t\leq T}|x^\mu(t)|=\mu^{\alpha}\mathbb{E}\sup\limits_{0\leq t\leq T}|B^H(t)-\mu^{H}w^\mu(t)|.$$
		Set $\bar{w}^\mu:=\mu^{H}w^\mu$, and one verifies immediately that
		\begin{equation}\label{sdew}
		d\bar{w}^\mu(t)=-\mu^{-1}\bar{w}^\mu(t)dt+dB^H(t),\quad \bar{w}^\mu(0)=0.
		\end{equation}
		The proposition is proved provided that we show
		\begin{equation}\label{esforbwe}
			\sup\limits_{0<\mu\leq 1}\mathbb{E}\sup\limits_{0\leq t\leq T}|\bar{w}^\mu(t)|<\infty,
		\end{equation}
		 and
		 \begin{equation}\label{esforL}
		 	\mathbb{E}\sup\limits_{0\leq t\leq T}|B^H(t)|<\infty.
		 \end{equation}
		Let us show (\ref{esforbwe}) first. Indeed, from (\ref{sdew}), we have
		$$\bar{w}^\mu(t)=\int_0^t e^{{-\mu}^{-1} (t-s)}dB^H(s),$$
		from the inequality (\ref{moment_inequality_fBm}), we have
            \begin{align*}
                \mathbb{E} \sup_{0 \leq t \leq T} \left| \bar{w}^\mu(t) \right| 
                &\leq C_H \sup_{0 \leq t \leq T} \left( \int_0^t \left| e^{{-\mu}^{-1} (t-s)} \right|^{\frac{1}{H}}  ds \right) \\
                &= C_H \sup_{0 \leq t \leq T} \left[ \mu H (1 - e^{-\frac{t}{\mu H}}) \right] \\
                &\leq C_H \cdot \mu H.
            \end{align*}
		We have proved (\ref{esforbwe}). Then let us prove (\ref{esforL}). Due to the fact that $B^H$ is k-Hölder   Continuous,where $0< k< H$, for all $0\leq t\leq T$, and $0< T< \infty$,
		\begin{equation*}
                \mathbb{E}\sup\limits_{0\leq t\leq T}|B^H(t)|< C \cdot T^k.
            \end{equation*}
		This finishes the proof.
	\end{proof}
	We give the moment estimate for $X^\mu$ with the help of Lemma \ref{thmforlinearEsup}.
	\begin{prop}\label{thmforEsup}
		For any $0 \leq t \leq T$ and $0 < \mu \leq 1$, let $X^\mu(t)$ denote the stochastic process defined by equation (\ref{wave}). Then the following uniform bound holds
        $$\sup\limits_{0<\mu\leq 1}\mathbb{E}\sup\limits_{0\leq t\leq T}|X^\mu(t)|<\infty.$$
	\end{prop}
	\begin{proof}
		Set $\rho^\mu:=X^\mu-x^\mu$. It follows from (\ref{wave}) and (\ref{linearwave}) that
		\begin{equation*}
                \mu\ddot{\rho}^\mu+\dot{\rho}^\mu=b(x^\mu+\rho^\mu),\quad \rho^\mu(0)=x_0,\dot{\rho}^\mu(0)=y_0,
            \end{equation*}
		and we can rewrite it as
		\begin{equation}\label{rdeslowfast}
			\begin{cases}
				\dot{\rho}^\mu=\xi^\mu,\rho^\mu(0)=x_0,\\
				\dot{\xi}^\mu=\mu^{-1}(-\xi^\mu+b(x^\mu+\rho^\mu)), \xi^\mu(0)=y_0.
			\end{cases}
		\end{equation}
		From (\ref{rdeslowfast}), we can solve $\xi^\mu$ analytically as
		$$\xi^\mu(t)=\mu^{-1}e^{-\mu^{-1}t}\int_{0}^{t} e^{\mu^{-1}s}(b(x^\mu(s)+\rho^\mu(s)))ds,$$
		so that by assumption \textbf{(A1)},
		\begin{align*}
                |\xi^\mu(t)| 
            &\leq \mu^{-1}e^{-\mu^{-1}t}\int_{0}^{t} e^{\mu^{-1}s}|b(x^\mu(s)+\rho^\mu(s))|ds \\
            &\lesssim \mu^{-1}e^{-\mu^{-1}t}\int_{0}^{t} e^{\mu^{-1}s}|x^\mu(s)|ds + \mu^{-1}e^{-\mu^{-1}t}\int_{0}^{t} e^{\mu^{-1}s}|\rho^\mu(s)|ds \\
            &\quad + \mu^{-1}e^{-\mu^{-1}t}\int_{0}^{t} e^{\mu^{-1}s}ds \\
            &\leq \mu^{-1}e^{-\mu^{-1}t}\int_{0}^{t} e^{\mu^{-1}s}\sup_{0\leq t\leq T}|x^\mu(t)|ds + \mu^{-1}e^{-\mu^{-1}t}\int_{0}^{t} e^{\mu^{-1}s}|\rho^\mu(s)|ds + 1 \\
            &\leq \sup_{0\leq t\leq T}|x^\mu(t)| + \mu^{-1}e^{-\mu^{-1}t}\int_{0}^{t} e^{\mu^{-1}s}|\rho^\mu(s)|ds + 1.
            \end{align*}
        
		Since $\dot{\rho}^\mu=\xi^\mu,$ we can continue our estimate as
		\begin{align*}
                |\rho^\mu(t)| 
            &\leq \int_{0}^{t}|\xi^\mu(s)|ds \\
            &\lesssim t\sup_{0\leq t\leq T}|x^\mu(t)| + \int_{0}^{t}\mu^{-1}e^{-\mu^{-1}s}\int_{0}^{s} e^{\mu^{-1}r}|\rho^\mu(r)|drds + 1 \\
            &= t\sup_{0\leq t\leq T}|x^\mu(t)| + \mu^{-1}\int_{0}^{t}\int_{0}^{s} e^{-\mu^{-1}(s-r)}|\rho^\mu(r)|drds + 1 \\
            &= t\sup_{0\leq t\leq T}|x^\mu(t)| + \mu^{-1}\int_{0}^{t}\int_{r}^{t} e^{-\mu^{-1}(s-r)}|\rho^\mu(r)|dsdr + 1 \\
            &= t\sup_{0\leq t\leq T}|x^\mu(t)| + \mu^{-1}\int_{0}^{t} e^{\mu^{-1}r}|\rho^\mu(r)|\int_{r}^{t}e^{-\mu^{-1}s}dsdr + 1 \\
            &= t\sup_{0\leq t\leq T}|x^\mu(t)| + \mu^{-1}\int_{0}^{t} e^{\mu^{-1}r}|\rho^\mu(r)|(\mu e^{-\mu^{-1}r}-\mu e^{-\mu^{-1}t})dr + 1 \\
            &\leq t\sup_{0\leq t\leq T}|x^\mu(t)| + \int_{0}^{t}|\rho^\mu(r)|dr + 1.
            \end{align*}
		Taking supremum with respect to $t$,
		\begin{align*}
                \sup_{0\leq t\leq T}|\rho^\mu(t)| 
            &\lesssim T\sup_{0\leq t\leq T}|x^\mu(t)| + \int_0^T|\rho^\mu(r)|dr + 1 \\
            &\lesssim T\sup_{0\leq t\leq T}|x^\mu(t)| + \int_0^T\sup_{0\leq s\leq t}|\rho^\mu(s)|dt + 1.
            \end{align*}
		Taking expectation and applying Lemma \ref{thmforlinearEsup},
		\begin{eqnarray*}
			\mathbb{E}\sup\limits_{0\leq t\leq T}|\rho^\mu(t)|\lesssim\int_0^T\mathbb{E}\sup_{0\leq s\leq t}|\rho^\mu(s)|dt+1.
		\end{eqnarray*}
		By Gronwall inequality,
		\begin{equation}\label{esforre}
			\mathbb{E}\sup\limits_{0\leq t\leq T}|\rho^\mu(t)|\lesssim1.
		\end{equation}
		Since $X^\mu=x^\mu+\rho^\mu,$ our proof is finished by combining (\ref{esforre}) and Lemma \ref{thmforlinearEsup}.
	\end{proof}
	To prove Theorem \ref{mainresult}, We begin with treating velocity part $Y^\mu$.
	\begin{lemma}
		For any $0 \leq t \leq T$ and $0 < \mu \leq 1$, the process $\bar{Y}^\mu_1(t)$ satisfies the following estimate
        $$\mathbb{E}\sup\limits_{0\leq t\leq T}|\mu^{-1}\int_{0}^{t}\bar{Y}^\mu_1(s)ds|\lesssim \mu.$$\label{esforby1}
	\end{lemma}
	\begin{proof}
		From (\ref{byeq})  
		\begin{equation*}
                \bar{Y}^\mu_1(t)=\mu y_0 e^{-\mu^{-1}t}.
            \end{equation*}
		As a consequence, for all $0\leq t\leq T$\,, 
		\begin{align*}
                \left|\mu^{-1}\int_{0}^{t} \bar{Y}^\mu_1(s)ds\right| = \left|\int_{0}^{t} y_0 e^{-\mu^{-1}s}ds\right| \leq |y_0|\int_{0}^{T}e^{-\mu^{-1}s}ds \leq \mu|y_0|.
            \end{align*}
		Taking supremum and expectation yields the result. 
	\end{proof}
	
	\begin{lemma}\label{esforby2}
		For any $0 \leq t \leq T$ and $0 < \mu \leq 1$, the process $\bar{Y}^\mu_2(t)$ satisfies the following uniform bound
        $$\sup\limits_{0<\mu\leq 1}\mathbb{E}\sup\limits_{0\leq t\leq T}|\bar{Y}^\mu_2(t)|<\infty.$$
	\end{lemma}	
	\begin{proof}
		From (\ref{byeq}),
		\begin{equation*}
                \bar{Y}^\mu_2(t)=\mu^{-1}e^{-\mu^{-1}t}\int_{0}^{t} e^{\mu^{-1} s}b(X^\mu(s))ds.
            \end{equation*}
		By assumption \textbf{(A1)}, for all $0\leq t\leq T$
		\begin{align*}
                    |\bar{Y}^\mu_2(t)| 
                &\lesssim \mu^{-1}e^{-\mu^{-1}t}\int_{0}^{t} e^{\mu^{-1} s}|X^\mu(s)|ds + \mu^{-1}e^{-\mu^{-1}t}\int_{0}^{t} e^{\mu^{-1} s}ds \\
                &\leq \mu^{-1}e^{-\mu^{-1}t}\sup_{0\leq t\leq T}|X^\mu(t)|\int_{0}^{t} e^{\mu^{-1} s}ds + 1 \\
                &\leq \sup_{0\leq t\leq T}|X^\mu(t)| + 1,
            \end{align*}
		and the result follows from Proposition \ref{thmforEsup} after taking expectation.
	\end{proof}
	
	\begin{lemma}\label{esforby3}
		For $0 \leq t \leq T$ and $0 < \mu \leq 1$, the process $\bar{Y}^\mu_3(t)$ satisfies the following uniform bound
        $$\sup\limits_{0<\mu\leq 1}\sup\limits_{0\leq t\leq T}\mathbb{E}|\bar{Y}^\mu_3(t)|<\infty.$$
	\end{lemma}	
	\begin{proof}
		From (\ref{byeq}),
		\begin{equation*}
                \bar{Y}^\mu_3(t)=\mu^{-H}e^{-\mu^{-1}t}\int_{0}^{t} e^{\mu^{-1} s}dB^H(s).
            \end{equation*}
		For all $0\leq t\leq T$
		\begin{align*}
                \mathbb{E}|\bar{Y}^\mu_3(t)| 
            &\leq C_{H} \mathbb{E}\left( \int_{0}^{t} \left| \mu^{-H} e^{\mu^{-1}(s - t)} \right|^{\frac{1}{H}}  ds \right) \\
            &= C_{H} \mu H \cdot \mu^{-1} \left( 1 - e^{-\frac{t}{\mu H}} \right) \\
            &\leq C_{H} H.
            \end{align*}
		The proof is complete.
	\end{proof}

  {\bf Proof of Theorem \ref{mainresult}.} 
	With the preparation made above, we are in a position to prove our approximate result. From (\ref{slowfast}) and~(\ref{decomposition}) we have 
	\begin{equation*}
            X^\mu(t)=x_0+\mu^{-1}\int_{0}^{t}\bar{Y}^\mu_1(s)ds+\int_{0}^{t}\bar{Y}^\mu_2(s)ds+\mu^{\alpha+H-1}\int_{0}^{t}\bar{Y}^\mu_3(s)ds.
        \end{equation*}
	By (\ref{byeq}), 
	\begin{equation*}
		\bar{Y}^\mu_2(t)=\int_{0}^{t} -\mu^{-1}[\bar{Y}^\mu_2(s)-b(X^\mu(s))]ds.
	\end{equation*}
	Combining the two equations above,
	\begin{equation}
		X^\mu(t)=x_0+\mu^{-1}\int_{0}^{t}\bar{Y}^\mu_1(s)ds+\int_{0}^{t} b(X^\mu(s))ds-\mu\bar{Y}^\mu_2(t)+\mu^{\alpha+H-1}\int_{0}^{t}\bar{Y}^\mu_3(s)ds.\label{newequationforx}
	\end{equation}
	From (\ref{ayeq}) and (\ref{newequationforx}) we deduce that 
	\begin{align}
		&|X^\mu(t)-\bar{X}^\mu(t)| \nonumber \\
		&\leq \Big|\mu^{-1}\int_{0}^{t}\bar{Y}^\mu_1(s)ds\Big| + \Big|\int_{0}^{t} b(X^\mu(s))-b(\bar{X}^\mu(s))ds\Big| \nonumber \\
		&\quad + \mu|\bar{Y}^\mu_2(t)| + \Big|\mu^{\alpha+H-1}\int_{0}^{t}\bar{Y}^\mu_3(s)ds-\mu^\alpha B^H(t)\Big| \nonumber \\
		&=: \sum_{k=1}^{4}J^\mu_k(t). \label{generaldifference}
	\end{align}
	As a consequence of Lemma \ref{esforby1}, 
	\begin{equation}
		\mathbb{E}\sup\limits_{0\leq t\leq T} J^\mu_1(t)\lesssim\mu. \label{esforJ1}
	\end{equation}
	By assumption \textbf{(A1)}, for all $0\leq t\leq T$,
	\begin{align*}
            \left|\int_{0}^{t} b(X^\mu(s))-b(\bar{X}^\mu(s))ds\right| 
        &\leq \int_{0}^{t} |b(X^\mu(s))-b(\bar{X}^\mu(s))|ds \\
        &\leq \int_{0}^{T} \sup_{0\leq r\leq s}|b(X^\mu(r))-b(\bar{X}^\mu(r))|ds \\
        &\lesssim \int_{0}^{T} \sup_{0\leq r\leq s}|X^\mu(r)-\bar{X}^\mu(r)|ds.
        \end{align*}
	Taking supremum and expectation and using Fubini theorem,
	\begin{equation}
		\mathbb{E}\sup\limits_{0\leq t\leq T} J^\mu_2(t)\lesssim \int_{0}^{T}\mathbb{E}\sup\limits_{0\leq s\leq t}|X^\mu(s)-\bar{X}^\mu(s)|dt. \label{esforJ2}
	\end{equation}
	From Lemma \ref{esforby2},
	\begin{equation}\label{esforJ3}
		\mathbb{E}\sup\limits_{0\leq t\leq T} J^\mu_3(t)\lesssim\mu.
	\end{equation}
	Lastly we deal with $J^\mu_4$\,. From (\ref{byeq}), 
	\begin{equation}
		\mu^H\bar{Y}^\mu_3(t)=B^H(t)-\mu^{H-1}\int_{0}^{t} \bar{Y}^\mu_3(s)ds,\label{newJ3}
	\end{equation}
	but one finds that $\bar{Y}^\mu_3$ coincides with $w^\mu$ defined in (\ref{defofw}), so that $\mu^H\bar{Y}^\mu_3$ coincides with $\bar{w}^\mu.$ Therefore, (\ref{esforbwe}) implies that there exists a constant $C>0$ such that for all $0<\mu\leq 1$, 
	$$\mathbb{E}\sup\limits_{0\leq t\leq T}|B^H(t)-\mu^{H-1}\int_{0}^{t} \bar{Y}^\mu_3(s)ds|\leq C.$$
	Multiplying both sides by $\mu^\alpha$,
	\begin{equation}
		\mathbb{E}\sup\limits_{0\leq t\leq T} J^\mu_4(t)=\mathbb{E}\sup\limits_{0\leq t\leq T}|\mu^\alpha B^H(t)-\mu^{\alpha+H-1}\int_{0}^{t} \bar{Y}^\mu_3(s)ds|\leq C\mu^\alpha.\label{esforJ4}
	\end{equation}
	Taking supremum and expectation on both sides of (\ref{generaldifference}) and combining (\ref{esforJ1})--(\ref{esforJ4}), 
	\begin{align}
            \mathbb{E}\sup_{0\leq t\leq T}|X^\mu(t)-\bar{X}^\mu(t)| 
        &\lesssim \int_{0}^{T}\mathbb{E}\sup_{0\leq s\leq t}|X^\mu(s)-\bar{X}^\mu(s)|dt + \mu + \mu^\alpha \nonumber \\
        &\leq \int_{0}^{T}\mathbb{E}\sup_{0\leq s\leq t}|X^\mu(s)-\bar{X}^\mu(s)|dt + \mu^\alpha,
        \end{align}
	where the last inequality follows from the fact that $0\leq\alpha<1$.
	By Gronwall inequality,
	$$\mathbb{E}\sup\limits_{0\leq t\leq T}|X^\mu(t)-\bar{X}^\mu(t)|\lesssim \mu^\alpha,$$
	and we finish the proof for part (i) of Theorem \ref{mainresult}.
	
	Let us turn to the proof of part (ii) of Theorem \ref{mainresult}, which is more subtle. In this case, from (\ref{generaldifference}), (\ref{esforJ1}), (\ref{esforJ3}), (\ref{newJ3}) and assumption \textbf{(A1)}, 
	$$\mathbb{E}|X^\mu(t)-\bar{X}^\mu(t)|\lesssim \int_{0}^{t}\mathbb{E}|X^\mu(s)-\bar{X}^\mu(s)|ds+\mu+\mathbb{E}|\mu^H\bar{Y}^\mu_3(t)|,$$
	which means 
	\begin{align}
\sup_{0\leq t\leq T}\mathbb{E}|X^\mu(t)-\bar{X}^\mu(t)| 
&\lesssim \int_0^{T}\mathbb{E}|X^\mu(t)-\bar{X}^\mu(t)|dt + \mu + \sup_{0\leq t\leq T}\mathbb{E}|\mu^H\bar{Y}^\mu_3(t)| \nonumber \\
&\leq \int_0^{T}\sup_{0\leq s\leq t}\mathbb{E}|X^\mu(s)-\bar{X}^\mu(s)|dt + \mu + \sup_{0\leq t\leq T}\mathbb{E}|\mu^H\bar{Y}^\mu_3(t)|.
\end{align}
	Then we have 
	$$\lim\limits_{\mu\rightarrow 0}\sup\limits_{0\leq t\leq T}\mathbb{E}|X^\mu(t)-\bar{X}^\mu(t)|=0$$  
	provided that 
	\begin{equation}
		\lim\limits_{\mu\rightarrow 0}\sup\limits_{0\leq t\leq T}\mathbb{E}|\mu^H\bar{Y}^\mu_3(t)|=0,
	\end{equation}
	which is an immediate consequence of Lemma \ref{esforby3}. The proof is complete. \qed
	
	\section{Large deviations}\label{sec large}
	
	In this section, we consider the case of Equation (\ref{wave}) with $\alpha = 0$, namely the classical Smoluchowski–Kramers approximation, and extend the equation with additive noise further to the case with small multiplicative noise,as described by the following equation

	\begin{equation}
	\begin{cases} 
			\mu \ddot{X}^{\mu,\varepsilon}(t) + \dot{X}^{\mu,\varepsilon}(t) = b({X}^{\mu,\varepsilon}(t)) + {\varepsilon}^{H}\sigma({X}^{\mu,\varepsilon}(t))\dot{B}^H(t),\\
		{X}^{\mu,\varepsilon}(0) = x_0,\quad \dot{X}^{\mu,\varepsilon}(0) = y_0.
	\end{cases}\label{original}
	\end{equation}
	Here \(\mu=\mu(\varepsilon)\to 0\) as \(\varepsilon\to 0.\) We aim to use the large deviation principle Lemma \ref{LDP} to establish the large deviation principle for \({X}^{\mu,\varepsilon}(t).\) To prove this result, we will first find the measurable maps \(\Gamma^{\varepsilon}\) and \(\Gamma^{0}\) and then show that they satisfy the conditions (i) and (ii) in Lemma \ref{LDP}.
	
	Let us first recall that the system (\ref{original}) can be rewritten as the following SDEs
	\begin{equation}
		\begin{cases} 
			d{{X}^{\mu,\varepsilon}}(t)={{Y}^{\mu,\varepsilon}}(t)dt,\\
			d{{Y}^{\mu,\varepsilon}}(t)=\tfrac{1}{\mu}b({X}^{\mu,\varepsilon}(t))dt-\tfrac{1}{\mu}{Y}^{\mu,\varepsilon}(t)dt+\tfrac{{\varepsilon}^{H}}{\mu}\sigma({X}^{\mu,\varepsilon}(t))dB^H(t),\\
			{X}^{\mu,\varepsilon}(0) = x_0,\quad {Y}^{\mu,\varepsilon}(0) = y_0.
		\end{cases}\label{fast}
	\end{equation}
	Similar to Lemma \ref{lemmaexistence}, (\ref{fast}) admits a unique strong solution. There exists a measurable map \(\Gamma^{\varepsilon}:{C}([0,T];\mathbb{R}^{d}) \rightarrow{C}([0,T];\mathbb{R}^{d})\) such that we have the representation \({X}^{\mu,\varepsilon}(t)=\Gamma^{\varepsilon}({\varepsilon}^{H}B^H(t)).\)
	
	Then for any \(h^{\varepsilon}\in\mathscr{A}_{N}\) let us define
	\[
	X^{\mu,\varepsilon,h^{\varepsilon}}:=\Gamma^{\varepsilon}({\varepsilon}^{H}B^H(\cdot)+\int_0^{\cdot}h^\varepsilon(s)ds),
	\]
	then \(X^{\mu,\varepsilon,h^{\varepsilon}}\) is the first part of solution of the following stochastic control problem
	\begin{equation}
		\begin{cases} 
			d{X^{\mu,\varepsilon,h^{\varepsilon}}}(t) = {Y^{\mu,\varepsilon,h^{\varepsilon}}}(t)dt,\\
			\begin{aligned}
				d{Y^{\mu,\varepsilon,h^{\varepsilon}}}(t) = {} & \tfrac{1}{\mu}b(X^{\mu,\varepsilon,h^{\varepsilon}}(t))dt-\tfrac{1}{\mu}Y^{\mu,\varepsilon,h^{\varepsilon}}(t)dt \\
				& + \tfrac{1}{\mu}\sigma(X^{\mu,\varepsilon,h^{\varepsilon}}(t))h^\varepsilon(t)dt+ \tfrac{{\varepsilon}^{H}}{\mu}\sigma(X^{\mu,\varepsilon,h^{\varepsilon}}(t))dB^H(t),
			\end{aligned}\\
			X^{\mu,\varepsilon,h^{\varepsilon}}(0) = x_0,\quad Y^{\mu,\varepsilon,h^{\varepsilon}}(0) = y_0.
		\end{cases}\label{pro}
	\end{equation}
	
	Moreover, intuitively, as \(\varepsilon\) tends to \(0\) in stochastic system (\ref{original}), the noise term vanishes, and in view of the theory of Smoluchowski-Kramers approximation we can get the following differential equation
	\begin{equation}\label{sim}
		\frac{d\bar{X}^0(t)}{dt} = b(\bar{X}^0(t)),\quad
		\bar{X}^0(0) = x_0.
	\end{equation}
	We mention that (\ref{sim}) admits a unique solution \(\bar{X}^0(t)\), and it is a deterministic path.
	
	The important part of the work is to find the measurable map \(\Gamma^{0}\). Since the ordinary differential equation (\ref{sim}) is, from a heuristically standpoint, a good approximation of the stochastic system (\ref{original}), as \(\varepsilon\) is small. We thus define the following skeleton equation
	\begin{equation}\label{ske}
		\frac{d\bar{X}^h(t)}{dt} = b(\bar{X}^h(t))+\sigma(\bar{X}^h(t))h(t),\quad
		\bar{X}^h(0) = x_0,
	\end{equation}
	where \(h\in\mathbb{H}\). Next, we give some important properties for \(\bar{X}^h(t)\).
	
	\begin{lemma}\label{cnt}
		Suppose that assumptions \textbf{(A1)}, \textbf{(A2)} hold. For any \(h\in\mathbb{H}\), equation (\ref{ske}) admits a unique solution \(\bar{X}^h(t)\) in \({C}([0,T];\mathbb{R}^{d})\). Moreover, for any \(N>0\), there exists a constant \(C_{N,T}\) such that
		\[
		\sup_{h\in S_{N}}\{\sup_{0\leq t\leq T}|\bar{X}^h(t)|\}\leq C_{N,T}.
		\]
	\end{lemma}
	
	The proof is straightforward and thus omitted. Furthermore, following from Lemma \ref{cnt}, it allows us to define a map \(\Gamma^{0}:{C}([0,T];\mathbb{R}^{d})\rightarrow{C}([0,T];\mathbb{R}^{d})\) by
	\begin{equation}
		\Gamma^{0}\left(\int_0^{\cdot}h(s)ds\right)=\bar{X}^{h}(\cdot).
	\end{equation}
	
	We next state the main result in this section.
	
	\begin{thm}\label{ldpresult}
		Suppose that assumptions \textbf{(A1)}, \textbf{(A2)} hold. Then \(\{{X}^{\mu,\varepsilon}(t)\}_{\varepsilon>0}\) satisfies a large deviation principle on \(C([0,T];\mathbb{R}^{d})\) with the rate function \(I\) given by
		\begin{eqnarray*}
            I(g)=\left\{\begin{array}{ll}
				\frac{1}{2}\|h\|_{\mathbb{H}}^2,\quad g\text{ is absolutely continuous,}\\
				+\infty,\quad\text{ for the rest of }C([0,T]).
			\end{array}\right.
		\end{eqnarray*}
	\end{thm}

	In the following, we devote to proving Theorem \ref{ldpresult} by proving condition (i) and (ii) in Lemma \ref{LDP} for the above mentioned maps \(\Gamma^{\varepsilon}\) and \(\Gamma^{0}\).
	
	We now prove that the condition (i) of Lemma \ref{LDP} holds.
	
	\begin{prop}\label{con}
		Under the assumptions \textbf{(A1)}, \textbf{(A2)}, let \(\{h^{\varepsilon}\}_{\varepsilon>0}\subset S_{N}\) for some \(N<\infty\) such that \(h^{\varepsilon}\) converges to element \(h\) in \(S_{N}\) as \(\varepsilon\to 0\), then
		\begin{eqnarray*}
			\lim_{\varepsilon\to 0}\sup_{0\leq t\leq T}\left|\Gamma^{0}\left(\int_0^{t}{h}^{\varepsilon}(s)ds\right)-\Gamma^{0}\left(\int_0^{t}h(s)ds\right)\right|=0.
		\end{eqnarray*}
		
	\end{prop}
	
	\begin{proof}
		Let \(\bar{X}^{h}\) be the solution of (\ref{ske}) and \(\bar{X}^{h^{\varepsilon}}\) be the solution of (\ref{ske}) with \(h\) replaced by \(h^{\varepsilon}\). By the definition of \(\Gamma^{0}\), \(\bar{X}^{h}(t)=\Gamma^{0}\left(\int_{0}^{t}h(s)ds\right)\) and \(\bar{X}^{h^{\varepsilon}}(t)=\Gamma^{0}\left(\int_{0}^{t}{h}^{\varepsilon}(s)ds\right)\). Note that \(\bar{X}^{h},\bar{X}^{h^{\varepsilon}}\in C([0,T];\mathbb{R}^{d})\).
		
		Firstly, we prove that \(\{\bar{X}^{h^{\varepsilon}}\}_{\varepsilon>0}\) is pre-compact in \(C([0,T];\mathbb{R}^{d})\). It suffices to show that \(\{\bar{X}^{h^{\varepsilon}}\}_{\varepsilon>0}\) is uniformly bounded and equi-continuous in \(C([0,T];\mathbb{R}^{d})\). It follows from the boundedness of \(b\), \(\sigma\) and \(h^{\varepsilon}\in S_{N}\), that
		\begin{eqnarray*}
		\sup_{\varepsilon>0}\sup_{0\leq t\leq T}|\bar{X}^{h^{\varepsilon}}(t)|:=C_{N,T}<\infty.
		\end{eqnarray*}
		For \(t>s\),
		\begin{align*}
               |\bar{X}^{h^{\varepsilon}}(t)-\bar{X}^{h^{\varepsilon}}(s)| &\leq \int_{s}^{t}|b(\bar{X}^{h^{\varepsilon}}(r))|dr + \left|\int_{s}^{t}\sigma(\bar{X}^{h^{\varepsilon}}(r)){h}^{\varepsilon}(r)dr\right| \\
               &\leq C|t-s| + C\left(\int_0^{T}|{h}^{\varepsilon}(r)|^{2}dr\right)^{\frac{1}{2}}|t-s|^{\frac{1}{2}} \\
               &\leq C_{N,T}|t-s|^{\frac{1}{2}}.
            \end{align*}
		Therefore, \(\{\bar{X}^{h^{\varepsilon}}\}_{\varepsilon>0}\) is pre-compact in \(C([0,T];\mathbb{R}^{d})\).
		
		Let \(l\) be a limit of some subsequence of \(\{\bar{X}^{h^{\varepsilon}}\}_{\varepsilon>0}\). We will show that \(l=\bar{X}^{h}\) completing the proof of the proposition. Without loss of generality, we simply assume
		\begin{equation}\label{lt}
		\lim_{\varepsilon\to 0}\sup_{0\leq t\leq T}|l_{t}-\bar{X}^{h^{\varepsilon}}(t)|=0.
		\end{equation}
		Using the Lipschitz property of \(b\) and (\ref{lt}) we see that
		\begin{eqnarray*}
			\int_0^{T}|b(\bar{X}^{h^{\varepsilon}}(t))-b(l_{t})|dt
			\leq C\int_0^{T}|\bar{X}^{h^{\varepsilon}}(t)-l_{t}|dt\\
			\leq C\sup_{0\leq t\leq T}|l_{t}-\bar{X}^{h^{\varepsilon}}(t)|\to 0,\quad \varepsilon\to 0.
		\end{eqnarray*}
		Hence, for each \(t\in[0,T]\),
		\begin{equation}\label{bxls}
		\int_0^{t}b(\bar{X}^{h^{\varepsilon}}(s))ds\to\int_0^{t}b(l_{s})ds,\quad \varepsilon\to 0.
		\end{equation}
		Similarly, by the Lipschitz condition of \(\sigma\), (\ref{lt}) and Cauchy-Schwarz inequality, we have
		\begin{eqnarray*}
			\int_0^{T}|\sigma(\bar{X}^{h^{\varepsilon}}(t))h^\varepsilon(t)-\sigma(l_{t})h^\varepsilon(t)|dt
			\leq C\int_0^{T}|\bar{X}^{h^{\varepsilon}}(t)-l_{t}||h^\varepsilon(t)|dt\\
			\leq C\sup\limits_{0\leq t\leq T}|l_{t}-\bar{X}^{h^{\varepsilon}}(t)|\left(\int_0^{T}|h^\varepsilon(t)|^{2}dt\right)^{\frac{1}{2}}\to 0,\quad \varepsilon\to 0.
		\end{eqnarray*}
		Moreover, since \(h^{\varepsilon}\to h\) on \(S_{N}\), we have
		\begin{eqnarray*}
		\int_0^{t}\sigma(l_{s})h^\varepsilon(s)ds\rightarrow\int_0^{t}\sigma(l_{s})h(s)ds,\quad \varepsilon\to 0.
		\end{eqnarray*}
		Therefore, one can derive that
		\begin{equation}\label{sigmals}
		\int_0^{t}\sigma(\bar{X}^{h^{\varepsilon}}(s))h^\varepsilon(s)ds\rightarrow\int_0^{t}\sigma(l_{s})h(s)ds,\quad \varepsilon\to 0.
		\end{equation}
		Recall that \(\bar{X}^{h^{\varepsilon}}\) be the solution of (\ref{ske}) with \(h\) replaced by \(h^{\varepsilon}\):
		\begin{eqnarray*}
		\bar{X}^{h^{\varepsilon}}(t)=x_{0}+\int_0^{t}b(\bar{X}^{h^{\varepsilon}}(s))ds+\int_0^{t}\sigma(\bar{X}^{h^{\varepsilon}}(s))h^\varepsilon(s)ds.
		\end{eqnarray*}
		Letting \(\varepsilon\to 0\) and taking into account (\ref{bxls}) and (\ref{sigmals}), we see that \(l\) is a solution to (\ref{ske}), and the uniqueness of the solutions of (\ref{ske}) implies that \(l=\bar{X}^{h}\), which completes the proof.
	\end{proof}

	We are now in the position to verify the condition (ii) of Lemma \ref{LDP}.
	
	\begin{prop}\label{proba}
		Under the assumptions \textbf{(A1)}, \textbf{(A2)}, let \(\{h^{\varepsilon}\}_{\varepsilon>0}\subset\mathscr{A}_{N}\) for some \(N<\infty\). Then for any \(\delta>0\), we have
		\begin{eqnarray*}
			\lim\limits_{\varepsilon\to 0}\mathbb{P}\left\{d\left(\Gamma^{\varepsilon}\left({\varepsilon}^{H} B^H(t)
			+\int_0^{t}h^\varepsilon(s)ds\right),\Gamma^{0}\left(\int_0^{t}h^\varepsilon(s)ds\right)\right)>\delta\right\}
			&=& 0.
		\end{eqnarray*}
	\end{prop}
	
	\begin{proof}
		Recall that \(X^{\mu,\varepsilon,h^{\varepsilon}}(t)=\Gamma^{\varepsilon}\left({\varepsilon}^{H}B^H(t)+\int_0^{t}h^\varepsilon(s)ds\right)\) and \(\bar{X}^{h^{\varepsilon}}(t)=\Gamma^{0}\left(\int_{0}^{t}h^\varepsilon(s)ds\right)\). Note that \(X^{\mu,\varepsilon,h^{\varepsilon}}(t)\) is the solution of (\ref{pro}) and \(\bar{X}^{h^{\varepsilon}}(t)\) is the solution of (\ref{ske}) with \(h\) replaced by \(h^{\varepsilon}\).
		
		Firstly, from (\ref{pro}) we can derive that
		
		\begin{align}\label{xtx}
                   X^{\mu,\varepsilon,h^{\varepsilon}}(t) 
                &= x_{0} + \int_0^{t} y_{0} e^{-\frac{s}{\mu}} ds 
                + \frac{1}{\mu} \int_0^{t} e^{-\frac{s}{\mu}} \left( \int_0^{s} e^{\frac{r}{\mu}} b(X^{\mu,\varepsilon,h^{\varepsilon}}(r)) dr \right) ds \nonumber \\
                &\quad + \frac{1}{\mu} \int_0^{t} e^{-\frac{s}{\mu}} \left( \int_0^{s} e^{\frac{r}{\mu}} \sigma(X^{\mu,\varepsilon,h^{\varepsilon}}(r)) h^{\varepsilon}(r) dr \right) ds \nonumber \\
                &\quad + \frac{\varepsilon^H}{\mu} \int_0^{t} e^{-\frac{s}{\mu}} \left( \int_0^{s} e^{\frac{r}{\mu}} \sigma(X^{\mu,\varepsilon,h^{\varepsilon}}(r)) dB^H(r) \right) ds.
            \end{align}
		
		We derive from (\ref{xtx}), after integrating by part,
		
		\begin{align}
                X^{\mu,\varepsilon,h^{\varepsilon}}(t) 
                &= x_{0} + \mu y_{0}(1 - e^{-\frac{t}{\mu}})
                + \int_0^{t} b(X^{\mu,\varepsilon,h^{\varepsilon}}(s)) ds
                - e^{-\frac{t}{\mu}} \int_0^{t} e^{\frac{s}{\mu}} b(X^{\mu,\varepsilon,h^{\varepsilon}}(s)) ds \nonumber \\
                &\quad + \int_0^{t} \sigma(X^{\mu,\varepsilon,h^{\varepsilon}}(s)) h^\varepsilon(s) ds
                - e^{-\frac{t}{\mu}} \int_0^{t} e^{\frac{s}{\mu}} \sigma(X^{\mu,\varepsilon,h^{\varepsilon}}(s)) h^\varepsilon(s) ds \nonumber \\
                &\quad + {\varepsilon}^H \int_0^{t} \sigma(X^{\mu,\varepsilon,h^{\varepsilon}}(s)) dB^H(s)
                - e^{-\frac{t}{\mu}} \int_0^{t} e^{\frac{s}{\mu}} {\varepsilon}^H \sigma(X^{\mu,\varepsilon,h^{\varepsilon}}(s)) dB^H(s).
            \end{align}
		Note that
		
		\begin{align*}
                X^{\mu,\varepsilon,h^{\varepsilon}}(t) - {\bar{X}}^{h^{\varepsilon}}(t) 
                &= \mu y_{0}(1 - e^{-\frac{t}{\mu}})
                + \int_0^{t} \left[b(X^{\mu,\varepsilon,h^{\varepsilon}}(s)) - b({\bar{X}}^{h^{\varepsilon}}(s))\right] ds \\
                &\quad + \int_0^{t} \left[\sigma(X^{\mu,\varepsilon,h^{\varepsilon}}(s))h^\varepsilon(s)- \sigma({\bar{X}}^{h^{\varepsilon}}(s))h^\varepsilon(s)\right] ds \\
                &\quad + {\varepsilon}^H \int_0^{t} \sigma(X^{\mu,\varepsilon,h^{\varepsilon}}(s)) dB^H(s)
                - e^{-\frac{t}{\mu}} \int_0^{t} e^{\frac{s}{\mu}} b(X^{\mu,\varepsilon,h^{\varepsilon}}(s)) ds \\
                &\quad - e^{-\frac{t}{\mu}} \int_0^{t} e^{\frac{s}{\mu}} \sigma(X^{\mu,\varepsilon,h^{\varepsilon}}(s)) h^\varepsilon(s) ds \\
                &\quad - e^{-\frac{t}{\mu}} \int_0^{t} e^{\frac{s}{\mu}} {\varepsilon}^H \sigma(X^{\mu,\varepsilon,h^{\varepsilon}}(s)) dB^H(s).
            \end{align*}
		Then we have
		
		\begin{align} \label{estimate}
                    \mathbb{E}\sup_{0\leq t\leq T}\left|X^{\mu,\varepsilon,h^{\varepsilon}}(t) - {\bar{X}}^{h^{\varepsilon}}(t)\right|^{2}
                    &\leq \mu^{2}C + C\mathbb{E}\int_0^{T}\left|b(X^{\mu,\varepsilon,h^{\varepsilon}}(s)) - b({\bar{X}}^{h^{\varepsilon}}(s))\right|^{2}ds\nonumber\\
                    &\quad + C\mathbb{E}\left|\int_0^{T}\left[\sigma(X^{\mu,\varepsilon,h^{\varepsilon}}(s)) - \sigma({\bar{X}}^{h^{\varepsilon}}(s))\right]h^\varepsilon(s)ds\right|^{2}\nonumber\\
                    &\quad + C\mathbb{E}\sup_{0\leq t\leq T}\left|\int_0^{t}{\varepsilon}^H\sigma(X^{\mu,\varepsilon,h^{\varepsilon}}(s))dB^H(s)\right|^{2}\nonumber\\
                    &\quad + C\mathbb{E}\sup_{0\leq t\leq T}\left|\int_0^{t}e^{\frac{s-t}{\mu}}b(X^{\mu,\varepsilon,h^{\varepsilon}}(s))ds\right|^{2}\nonumber\\
                    &\quad + C\mathbb{E}\sup_{0\leq t\leq T}\left|\int_0^{t}e^{\frac{s-t}{\mu}}\sigma(X^{\mu,\varepsilon,h^{\varepsilon}}(s)){h}^{\varepsilon}(s)ds\right|^{2}\nonumber\\
                    &\quad + C\mathbb{E}\sup_{0\leq t\leq T}\left|\int_0^{t}e^{\frac{s-t}{\mu}}\sigma(X^{\mu,\varepsilon,h^{\varepsilon}}(s))dB^H(s)\right|^{2}.
            \end{align}
		By assumption \textbf{(A3)} and the maximal inequality (\ref{maximal_inequality}), we have
		
		\begin{align} \label{eq:stochastic_integral_estimate}
                & C\mathbb{E}\sup_{0\leq t\leq T}\left|\int_0^{t}{\varepsilon}^H\sigma(X^{\mu,\varepsilon,h^{\varepsilon}}(s))dB^H(s)\right|^{2} \nonumber\\
                &\leq C \left[ \int_0^T \left| \mathbb{E}(\varepsilon^H \sigma(X^{\mu,\varepsilon,h^{\varepsilon}}(s))) \right|^2 ds 
                + \mathbb{E}\int_0^T \left( \int_0^T |D_{s} (\varepsilon^H \sigma(X^{\mu,\varepsilon,h^{\varepsilon}}(r)))|^{\frac{1}{H}}ds \right)^{2H} dr \right] \nonumber\\
                &\leq C \left( \varepsilon^{2H} K^2 T + \int_0^T \varepsilon^{2H} K^2 T^{2H} dr \right) 
                = C_{K,T} \varepsilon^{2H}.
            \end{align}
		Similarly,
		\begin{align} \label{eq:exponential_integral_estimate}
                & C\mathbb{E}\sup_{0\leq t\leq T}\left|\int_0^{t}e^{\frac{s-t}{\mu}}\sigma(X^{\mu,\varepsilon,h^{\varepsilon}}(s))dB^H(s)\right|^{2} \nonumber\\
                &\leq C \sup_{0 \leq t \leq T} \biggl[ \int_0^t \left| \mathbb{E}(e^{\frac{s-t}{\mu}} \sigma(X^{\mu,\varepsilon,h^{\varepsilon}}(s)) \right|^2 ds \nonumber\\
                &\quad + \mathbb{E}\int_0^t \left( \int_0^t |D_{s} (e^{\frac{r-t}{\mu}} \sigma(X^{\mu,\varepsilon,h^{\varepsilon}}(r))|^{\frac{1}{H}}ds \right)^{2H} dr \biggr] \nonumber\\
                &\leq C \sup_{0 \leq t \leq T} \left[ \frac{\mu}{2} (1-e^{-\frac{2t}{\mu}}) K^2 + \int_0^t K^2 e^{\frac{2(r-t)}{\mu}} t^{2H} dr \right] \nonumber\\
                &\leq C K^2 \frac{\mu}{2} \left(1 + T^{2H}\right) = C_{K,T} \mu.
            \end{align}
		
		Then using the assumption \textbf{(A1)}, \textbf{(A2)}, and Cauchy-Schwarz inequality, and substituting \eqref{eq:stochastic_integral_estimate} and \eqref{eq:exponential_integral_estimate} into \eqref{estimate}, we have
		\begin{align*}
                \mathbb{E}\sup_{0\leq t\leq T}&\left|X^{\mu,\varepsilon,h^{\varepsilon}}(t) - {\bar{X}}^{h^{\varepsilon}}(t)\right|^{2} \\
                &\leq \mu^2 C + C\mathbb{E}\int_0^{T}\left|X^{\mu,\varepsilon,h^{\varepsilon}}(s) - {\bar{X}}^{h^{\varepsilon}}(s)\right|^{2}ds \\
                &\quad + C\mathbb{E}\left(\int_0^{T}\left|\sigma(X^{\mu,\varepsilon,h^{\varepsilon}}(s)) - \sigma({\bar{X}}^{h^{\varepsilon}}(s))\right|^{2}ds\right)\left(\int_0^{T}|{h}^{\varepsilon}(s)|^{2}ds\right)\\
                &\quad + C_{K,T} \varepsilon^{2H} + C  \mathbb{E}  \sup_{0 \leq t \leq T} \left| \int_0^t e^{\frac{s-t}{\mu}}ds \right|^2 \\
                &\quad + C  \mathbb{E} \sup_{0 \leq t \leq T} \left| \int_0^t e^{\frac{s-t}{\mu}}{h}^{\varepsilon}(s)ds \right|^2 + C_{K,T} \mu \\
                &\leq \mu^2C + C\mathbb{E}\int_0^{T}\left|X^{\mu,\varepsilon,h^{\varepsilon}}(s) - {\bar{X}}^{h^{\varepsilon}}(s)\right|^{2}ds\\
                &\quad + CN\cdot \mathbb{E}\int_0^{T}\left|X^{\mu,\varepsilon,h^{\varepsilon}}(s) - {\bar{X}}^{h^{\varepsilon}}(s)\right|^{2}ds +  C_{K,T} \varepsilon^{2H}+C\mu^{2} \\
                &\quad + C\mathbb{E}\sup_{0\leq t\leq T}\left(\int_0^{t}e^{\frac{2(s-t)}{\mu}}ds\right)\left(\int_0^{T}|{h}^{\varepsilon}(s)|^{2}ds\right) + C_{K,T} \mu \\
                &\leq \mu^2C + C_N\mathbb{E}\int_0^{T}\left|X^{\mu,\varepsilon,h^{\varepsilon}}(s) - {\bar{X}}^{h^{\varepsilon}}(s)\right|^{2}ds + C_{K,T}\varepsilon^{2H} \\
                &\quad + C\mu^2 + C\mu N + C_{K,T} \mu.
            \end{align*}
		
		We derive that
		\begin{align*}
                \mathbb{E}\sup_{0\leq t\leq T}|X^{\mu,\varepsilon,h^{\varepsilon}}(t)-\bar{X}^{h^{\varepsilon}}(t)|^{2} 
                &\leq C\int_0^{T}\mathbb{E}\sup_{0\leq r\leq s}|X^{\mu,\varepsilon,h^{\varepsilon}}(r)-\bar{X}^{h^{\varepsilon}}(r)|^{2}ds \\
                &\quad + (\mu+\varepsilon^{2H})C_{N,K,T}.
            \end{align*}
		Then the Gronwall's inequality yields
		\begin{align*}
                 \mathbb{E}\sup_{0\leq t\leq T}|X^{\mu,\varepsilon,h^{\varepsilon}}(t)-\bar{X}^{h^{\varepsilon}}(t)|^{2} 
                 &\leq (\mu+\varepsilon^{2H})C_{N,K,T}\cdot e^{CT}.
            \end{align*}
		Applying the Chebyshev's inequality, for any \(\delta>0\) we have
		
		\begin{align*}
                &\mathbb{P}\left\{d\left(\Gamma^{\varepsilon}\left({\varepsilon}^HB^H(\cdot)+\int_0^{\cdot}{h}^{\varepsilon}(s)ds\right),\Gamma^{0}\left(\int_0^{\cdot}{h}^{\varepsilon}(s)ds\right)\right)>\delta\right\} \\
                &= \mathbb{P}\left\{\sup_{0\leq t\leq T}|X^{\mu,\varepsilon,h^{\varepsilon}}(t)-\bar{X}^{h^{\varepsilon}}(t)|>\delta\right\} \\
                &\leq \frac{\mathbb{E} \sup_{0\leq t\leq T}|X^{\mu,\varepsilon,h^{\varepsilon}}(t)-\bar{X}^{h^{\varepsilon}}(t)|^{2}}{\delta^{2}}\to 0,
            \end{align*}
		as \(\varepsilon\to 0\), \(\mu=\mu(\varepsilon)\to 0\).
		The proof is complete.
	\end{proof}

	{\bf Proof of Theorem \ref{ldpresult}.}
	Combining Propositions \ref{con} and \ref{proba}, it follows by Lemma \ref{LDP} that \(\{X^{\mu,\varepsilon}\}_{\varepsilon>0}\) satisfies a large deviation principle in \(C([0,T];\mathbb{R}^{d})\) with the rate function \(I\) given by
		\begin{align*}
                 I(g) = \inf_{\substack{h\in\mathbb{H},\\ 
                 {g}(t)=x_{0}+\int_0^{t}b(g(s))ds+\int_0^{t}\sigma(g(s))h(s)ds}}  
                 \left\{\frac{1}{2}\|h\|_{\mathbb{H}}^2\right\}.
            \end{align*}
            This completes the proof.	\qed
	
\section{Moderate deviations}\label{sec mod}

In this section we investigate deviations of \(X^{\mu,\varepsilon}(t)\) from the deterministic solution \(\bar{X}^0(t)\) defined by (\ref{sim}), as \(\varepsilon\) tends to \(0\), that is, the asymptotic behavior of the trajectory,
\begin{eqnarray*}
	\eta^{\varepsilon}(t) &=& \frac{X^{\mu,\varepsilon}(t) - \bar{X}^0(t)}{{\varepsilon}^H\lambda(\varepsilon)},
\end{eqnarray*}
where \(\lambda(\varepsilon)\) is some deviation scale which strongly influences the asymptotic behavior of \(\eta^{\varepsilon}(t)\).

Note that in Section \ref{sec large} we have studied the large deviation which solves the above deviation problem when \(\lambda(\varepsilon) = 1/{\varepsilon}^H\). If \(\lambda(\varepsilon) = 1\), we are in a position to study central limit theorem which is another subject we will discuss in the forthcoming paper.

To fill in the gap between the central limit theorem scale and the large deviations scale, we will study the moderate deviations for \(X^{\mu,\varepsilon}\), that is when the deviation scale satisfies
\begin{eqnarray*}
	\lambda(\varepsilon) \to \infty, \quad {\varepsilon}^H\lambda(\varepsilon) \to 0, \quad \text{as} \quad \varepsilon \to 0.
\end{eqnarray*}

We furthermore introduce the following assumption.

\textbf{(A4)}. The coefficient \(b(x)\) is differentiable with respect to \(x\), and its derivative functions satisfy
\[
\| \nabla b(x) \| \leq C, \quad \text{and} \quad \| \nabla b(x) - \nabla b(y) \| \leq C |x - y|.
\]

As an immediate consequence of Theorem 4.6 in \cite{liu2023large}, we present the following lemma to prove the moderate deviation principle for \(X^{\mu,\varepsilon}\) in our current setting. For a specific proof we also refer the reader to
\cite[Theorem 9.9]{budhiraja2019analysis} and \cite[Theorem 3.2]{matoussi2021large}.

\begin{lemma}\label{mdp}
	For any \(\varepsilon > 0\), let \(\Upsilon^{\varepsilon}\) be a measurable mapping from \(C([0,T];\mathbb{R}^d)\) into \(C([0,T];\mathbb{R}^d)\). Suppose that \(\{\Upsilon^{\varepsilon}\}_{\varepsilon > 0}\) satisfies the following assumptions: there exists a measurable map \(\Upsilon^0 : C([0,T];\mathbb{R}^d) \to C([0,T];\mathbb{R}^d)\) such that
	
	\begin{itemize}
		\item[(i)] Let \(\{h^{\varepsilon}\}_{\varepsilon > 0} \subset S_N\) for some \(N < \infty\) such that \(h^{\varepsilon}\) converges to element \(h\) in \(S_N\) as \(\varepsilon \to 0\), then \(\Upsilon^0 \left( \int_0^{\cdot} {h}^{\varepsilon}(s) ds \right)\) converges to \(\Upsilon^0 \left( \int_0^{\cdot} h(s) ds \right)\) in \(C([0,T];\mathbb{R}^d)\);
		
		\item[(ii)] Let \(\{h^{\varepsilon}\}_{\varepsilon > 0} \subset \mathscr{A}_N\) for some \(N < \infty\). For any \(\delta > 0\), we have
		\[
		\lim_{\varepsilon \to 0} \mathbb{P} \left\{ d \left( \Upsilon^{\varepsilon} \left( \frac{1}{\lambda(\varepsilon)} B^H(\cdot) + \int_0^{\cdot} {h}^{\varepsilon}(s) ds \right), \Upsilon^0 \left( \int_0^{\cdot} {h}^{\varepsilon}(s) ds \right) \right) > \delta \right\} = 0,
		\]
		where \(d(\cdot,\cdot)\) denotes the metric in the space \(C([0,T];\mathbb{R}^d)\).
	\end{itemize}
	
	Let \(X^{\varepsilon} = \Upsilon^{\varepsilon} \left( \frac{1}{\lambda(\varepsilon)} B^H(\cdot) \right)\), then the family \(\{X^{\varepsilon}\}_{\varepsilon > 0}\) satisfies a large deviation principle in \(C([0,T];\mathbb{R}^d)\) with speed \(\lambda^{-2}(\varepsilon)\) and rate function \(I\) given by
	
	\[
	I(g) = \inf_{h \in \mathbb{H},\; g = \Upsilon^0 \left( \int_0^{\cdot} h(s) ds \right)} \left\{ \frac{1}{2} \|h\|_{\mathbb{H}}^2 \right\}, \quad g \in C([0,T];\mathbb{R}^d),
	\]
	with \(\inf \emptyset = \infty\) by convention.
\end{lemma}

In analogy to the case demonstrating large deviation problems, we establish the moderate deviation principle for \(X^{\mu,\varepsilon}(t)\) by first finding the measurable maps \(\Upsilon^{\varepsilon}\) and \(\Upsilon^0\) and then proving that they satisfy the conditions (i) and (ii) in Lemma \ref{mdp}.

On the one hand, from (\ref{fast}) and(\ref{sim}), \(\eta^{\varepsilon}(t) = \frac{X^{\mu,\varepsilon}(t) - \bar{X}^{0}(t)}{\varepsilon^{H} \lambda(\varepsilon)}\) satisfies
\begin{align*}
        \eta^\varepsilon(t) 
        &= \frac{\mu}{\varepsilon^H \lambda(\varepsilon)} y_0 (1 - e^{-\frac{t}{\mu}})  + \frac{1}{\varepsilon^H \lambda(\varepsilon)} \int_0^t \left[ b(\bar{X}^{0}(s) + \varepsilon^H \lambda(\varepsilon) \eta^\varepsilon(s)) - b(\bar{X}^{0}(s)) \right] ds \\
        &\quad - \frac{1}{\varepsilon^H \lambda(\varepsilon)} e^{-\frac{t}{\mu}} \int_0^t e^{\frac{s}{\mu}} b(\bar{X}^{0}(s) + \varepsilon^H \lambda(\varepsilon) \eta^\varepsilon(s)) ds \\
        &\quad + \frac{1}{\lambda(\varepsilon)} \int_0^t \sigma(\bar{X}^{0}(s) + \varepsilon^H \lambda(\varepsilon) \eta^\varepsilon(s)) \, dB^H(s) \\
        &\quad - \frac{1}{\lambda(\varepsilon)} e^{-\frac{t}{\mu}} \int_0^t e^{\frac{s}{\mu}} \sigma(\bar{X}^{0}(s) + \varepsilon^H \lambda(\varepsilon) \eta^\varepsilon(s)) \, dB^H(s).
\end{align*}

There exists a measurable map \(\Upsilon^{\varepsilon} : C([0,T];\mathbb{R}^d) \to C([0,T];\mathbb{R}^d)\) such that we have the representation \(\eta^{\varepsilon}(t) = \Upsilon^{\varepsilon} \left( \frac{1}{\lambda(\varepsilon)} B^H(t) \right)\).

Then for any \(h^{\varepsilon} \in \mathscr{A}_N\) let us define

\begin{align*}
\eta^{\varepsilon,h^{\varepsilon}} = \Upsilon^{\varepsilon} \left( \frac{1}{\lambda(\varepsilon)} B^H(\cdot) 
+ \int_0^{\cdot} {h}^{\varepsilon}(s) \, ds \right),
\end{align*}

then \(\eta^{\varepsilon,h^{\varepsilon}}\) satisfies 
\begin{align} \label{eta}
        \eta^{\varepsilon,h^\varepsilon}(t) 
        &= \frac{\mu}{\varepsilon^H \lambda(\varepsilon)} y_0 (1 - e^{-\frac{t}{\mu}}) \nonumber  + \frac{1}{\varepsilon^H \lambda(\varepsilon)} \int_0^t \left[ b(\bar{X}^{0}(s) + \varepsilon^H \lambda(\varepsilon)\eta^{\varepsilon,h^\varepsilon}(s)) - b(\bar{X}^{0}(s)) \right] ds \nonumber \\
        &\quad - \frac{1}{\varepsilon^H \lambda(\varepsilon)} e^{-\frac{t}{\mu}} \int_0^t e^{\frac{s}{\mu}} b(\bar{X}^{0}(s) + \varepsilon^H \lambda(\varepsilon)\eta^{\varepsilon,h^\varepsilon}(s)) ds \nonumber \\
        &\quad + \int_0^t \sigma(\bar{X}^{0}(s) + \varepsilon^H \lambda(\varepsilon)\eta^{\varepsilon,h^\varepsilon}(s)) {h}^{\varepsilon}(s) \, ds \nonumber \\
        &\quad - e^{-\frac{t}{\mu}} \int_0^t e^{\frac{s}{\mu}} \sigma(\bar{X}^{0}(s) + \varepsilon^H \lambda(\varepsilon)\eta^{\varepsilon,h^\varepsilon}(s)) {h}^{\varepsilon}(s) \, ds \nonumber \\
        &\quad + \frac{1}{\lambda(\varepsilon)} \int_0^t \sigma(\bar{X}^{0}(s) + \varepsilon^H \lambda(\varepsilon) \eta^{\varepsilon,h^\varepsilon}(s)) \, dB^H(s) \nonumber \\
        &\quad - \frac{1}{\lambda(\varepsilon)} e^{-\frac{t}{\mu}} \int_0^t e^{\frac{s}{\mu}} \sigma(\bar{X}^{0}(s) + \varepsilon^H \lambda(\varepsilon) \eta^{\varepsilon,h^\varepsilon}(s)) \, dB^H(s).
\end{align}

On the other hand, we define the following skeleton equation:

\begin{equation}\label{skel}
\frac{d\tilde{X}^h(t)}{dt} = \nabla b(\bar{X}^0(t))\tilde{X}^h(t) + \sigma(\bar{X}^0(t)) h(t), \quad \tilde{X}^h(0) = 0,
\end{equation}
where \(h \in \mathbb{H}\). Then we have the following properties for \(\tilde{X}^h\).

\begin{lemma}\label{skeletonprop}
	Suppose that assumptions \textbf{(A1)}--\textbf{(A4)}. For any \(h \in \mathbb{H}\), equation (\ref{skel}) admits a unique solution \(\tilde{X}^h(t)\) in \(C([0,T];\mathbb{R}^d)\). Moreover, for any \(N > 0\), there exists a constant \(C_{N,T}\) such that
	
	\begin{equation}\label{bound}
	\sup_{h \in S_N} \left\{ \sup_{0 \leq t \leq T} |\tilde{X}^h(t)| \right\} \leq C_{N,T}.
	\end{equation}
\end{lemma}

The proof is straightforward, we thus omit it. Furthermore, following from Lemma \ref{skeletonprop}, it allows us to define a map \(\Upsilon^0 : C([0,T];\mathbb{R}^d) \to C([0,T];\mathbb{R}^d)\) by

\begin{equation}
\Upsilon^0 \left( \int_0^{\cdot} h(s) ds \right) = \tilde{X}^h(\cdot).
\end{equation}

Now we state the main result in this section.

\begin{thm}\label{mdpresult}
	Suppose that assumptions \textbf{(A1)}--\textbf{(A4)} hold and \(\lim_{\varepsilon \to 0} \frac{\mu(\varepsilon)}{{\varepsilon}^H} = 0\). Then \(\{\eta^{\varepsilon}(t)\}_{\varepsilon > 0}\) satisfies a large deviation principle on \(C([0,T];\mathbb{R}^d)\) with speed \(\lambda^{-2}(\varepsilon)\) and rate function \(\tilde{I}\) given by:
	
	\begin{align*}
            \tilde{I}(g) = \inf_{h \in \mathbb{H},\; g = \Upsilon^0 \left( \int_0^{\cdot} h(s) ds \right)} 
            \left\{ \frac{1}{2} \|h\|_{\mathbb{H}}^2\right\},\quad g \in C([0,T];\mathbb{R}^d),
        \end{align*}
	where \(\inf \emptyset = \infty\) by convention and \(\Upsilon^0 \left( \int_0^{\cdot} h(s) ds \right) = \tilde{X}^h\) satisfies the skeleton equation(\ref{skel}).
\end{thm}

To complete the proof of Theorem \ref{mdpresult}, it is sufficient to verify the conditions (i) and (ii) in Lemma \ref{mdp} for the above mentioned maps \(\Upsilon^{\varepsilon}\) and \(\Upsilon^0\). The verification of (i) and (ii) will be given in Propositions \ref{mdpcon} and \ref{mdpproba} respectively.

We now proceed to the proof of condition (i).

\begin{prop}\label{mdpcon}
	Under the assumptions \textbf{(A1)}--\textbf{(A4)}, let \(\{h^{\varepsilon}\}_{\varepsilon>0}\subset S_{N}\) for some \(N<\infty\) such that \(h^{\varepsilon}\) converges to element \(h\) in \(S_{N}\) as \(\varepsilon\to 0\), then \(\Upsilon^{0}\left(\int_0^{\cdot}{h}^{\varepsilon}(s)ds\right)\) converges to \(\Upsilon^{0}\left(\int_0^{\cdot}h(s)ds\right)\) in \(C([0,T];\mathbb{R}^{d})\).
\end{prop}

\begin{proof}
	Let \(\tilde{X}^h=\Upsilon^{0}\left(\int_{0}^{\cdot}h(s)ds\right)\) and \(\tilde{X}^{h^{\varepsilon}}=\Upsilon^{0}\left(\int_{0}^{\cdot}{h}^{\varepsilon}(s)ds\right)\) be the corresponding solutions to the skeleton equation(\ref{skel}). We need to prove the following result:
	\begin{align*}
            \lim_{\varepsilon\to 0}\sup_{t\in[0,T]}|{\tilde{X}^{h^\varepsilon}(t)}-{\tilde{X}^h}(t)| = 0.
        \end{align*}
	
	The proof is similar to that of Proposition \ref{con} and we just give a sketch here. We first show that \(\{\tilde{X}^{h^{\varepsilon}}\}_{\varepsilon>0}\) is pre-compact in \(C([0,T];\mathbb{R}^{d})\). (\ref{bound}) implies that \(\{\tilde{X}^{h^{\varepsilon}}\}_{\varepsilon>0}\) is uniformly bounded, i.e.,
	\begin{equation}\label{boundedness}
	\sup_{\varepsilon}\sup_{t\in[0,T]}|\tilde{X}^{h^\varepsilon}(t)|\leq C_{N,T}.
	\end{equation}
	
	For any \(s,t\in[0,T]\) with \(s<t\), by assumptions \textbf{(A1)}--\textbf{(A4)} and (\ref{boundedness}),
	\begin{align*}
            |\tilde{X}^{h^\varepsilon}(t)-\tilde{X}^{h^\varepsilon}(s)| 
            &\leq C_{N,T}\int_{s}^{t}|\nabla b(\bar{X}^{0}(r))|dr 
            + \int_{s}^{t}|\sigma(\bar{X}^{0}(r)){h}^{\varepsilon}(r)|dr \\
            &\leq C_{N,T}|t-s|^{\frac{1}{2}}.
        \end{align*}
	
	Therefore, \(\{\tilde{X}^{h^{\varepsilon}}\}_{\varepsilon>0}\) is pre-compact in \(C([0,T];\mathbb{R}^{d})\).
	
	Let \(\tilde{l}\) be any limit of some subsequence of \(\{\tilde{X}^{h^{\varepsilon}}\}_{\varepsilon>0}\) in \(C([0,T],\mathbb{R}^{d})\). Using similar arguments as in the proof of Proposition \ref{con}, we can show \(\tilde{l}=\tilde{X}^h\) which completes the proof.
\end{proof}

To verify condition (ii) in Lemma \ref{mdp}, we also need the following estimate.

\begin{lemma}\label{mdpest}
	Suppose that assumptions \textbf{(A1)}--\textbf{(A4)} hold and \(\lim_{\varepsilon\to 0}\frac{\mu}{{\varepsilon}^H}=0\). Let \({\eta^{\varepsilon,h^{\varepsilon}}}(t)\) be the solution of the controlled equation, then there exists a constant \(C_{N,K,T}>0\) such that
	
	\begin{eqnarray}\label{etabound}
			\sup_{\varepsilon}\mathbb{E}\sup_{0\leq t\leq T}|{\eta^{\varepsilon,h^{\varepsilon}}}(t)|^{2}\leq C_{N,K,T}.
	\end{eqnarray}
	
\end{lemma}

\begin{proof}
	Note that \({\eta^{\varepsilon,h^{\varepsilon}}}(t)\) has the representation (\ref{eta})
	\begin{align*}
            \eta^{\varepsilon,h^\varepsilon}(t) 
            &= \frac{\mu}{\varepsilon^H \lambda(\varepsilon)} y_0 (1 - e^{-\frac{t}{\mu}})  + \frac{1}{\varepsilon^H \lambda(\varepsilon)} \int_0^t \left[ b(\bar{X}^{0}(s) + \varepsilon^H \lambda(\varepsilon)\eta^{\varepsilon,h^\varepsilon}(s)) - b(\bar{X}^{0}(s)) \right] ds \\
            &\quad - \frac{1}{\varepsilon^H \lambda(\varepsilon)} e^{-\frac{t}{\mu}} \int_0^t e^{\frac{s}{\mu}} b(\bar{X}^{0}(s) + \varepsilon^H \lambda(\varepsilon)\eta^{\varepsilon,h^\varepsilon}(s)) ds \\
            &\quad + \int_0^t \sigma(\bar{X}^{0}(s) + \varepsilon^H \lambda(\varepsilon)\eta^{\varepsilon,h^\varepsilon}(s)) {h}^{\varepsilon}(s) ds \\
            &\quad - e^{-\frac{t}{\mu}} \int_0^t e^{\frac{s}{\mu}} \sigma(\bar{X}^{0}(s) + \varepsilon^H \lambda(\varepsilon)\eta^{\varepsilon,h^\varepsilon}(s)) {h}^{\varepsilon}(s) ds \\
            &\quad + \frac{1}{\lambda(\varepsilon)} \int_0^t \sigma(\bar{X}^{0}(s) + \varepsilon^H \lambda(\varepsilon) \eta^{\varepsilon,h^\varepsilon}(s)) dB^H(s) \\
            &\quad - \frac{1}{\lambda(\varepsilon)} e^{-\frac{t}{\mu}} \int_0^t e^{\frac{s}{\mu}} \sigma(\bar{X}^{0}(s) + \varepsilon^H \lambda(\varepsilon) \eta^{\varepsilon,h^\varepsilon}(s)) dB^H(s).
        \end{align*}
	By assumption \textbf{(A1)}, we have
	\begin{align} \label{eq:expectation_estimate}
            &\left|\frac{1}{{\varepsilon}^H\lambda(\varepsilon)}\mathbb{E}\sup_{0\leq t\leq T}\int_{0}^{t}\left[b(\bar{X}^{0}(s) + \varepsilon^H \lambda(\varepsilon)\eta^{\varepsilon,h^\varepsilon}(s))-b(\bar{X}^{0}(s))\right]ds\right|^{2} \nonumber\\
            &\leq C\mathbb{E}\int_{0}^{T}|{\eta^{\varepsilon,h^{\varepsilon}}}(s)|^{2}ds 
            \leq C\int_{0}^{T}\mathbb{E}\sup_{0\leq r\leq s}|{\eta^{\varepsilon,h^{\varepsilon}}}(r)|^{2}ds.
        \end{align}
	By assumption \textbf{(A3)} and the maximal inequality (\ref{maximal_inequality}), we have
	
	\begin{align} \label{eq:fractional_integral_estimate}
                &\mathbb{E}\sup_{0\leq t\leq T}\left|\frac{1}{\lambda(\varepsilon)} \int_0^t \sigma(\bar{X}^{0}(s) + \varepsilon^H \lambda(\varepsilon) \eta^{\varepsilon,h^\varepsilon}(s)) dB^H(s)\right|^{2} \nonumber\\
                &\leq C \frac{1}{\lambda^2(\varepsilon)}\left[ \int_0^T \left| \mathbb{E}\sigma(\bar{X}^{0}(s) + \varepsilon^H \lambda(\varepsilon) \eta^{\varepsilon,h^\varepsilon}(s)) \right|^{2} ds \right. \nonumber\\
                &\quad \left. + \mathbb{E}\int_0^T \left( \int_0^T |D_{s}  \sigma(\bar{X}^{0}(r) + \varepsilon^H \lambda(\varepsilon) \eta^{\varepsilon,h^\varepsilon}(r))|^{\frac{1}{H}}ds \right)^{2H} dr \right] \nonumber\\
                &\leq C \frac{1}{\lambda^2(\varepsilon)} \left(K^2 T + \int_0^T K^2 T^{2H} dr \right) 
                = C_{K,T} \frac{1}{\lambda^2(\varepsilon)}.
        \end{align}
	Similarly,
	\begin{align} \label{eq:exponential_fractional_estimate}
            &\mathbb{E}\sup_{0\leq t\leq T}\left|\frac{1}{\lambda(\varepsilon)} e^{-\frac{t}{\mu}} \int_0^t e^{\frac{s}{\mu}} \sigma(\bar{X}^{0}(s) + \varepsilon^H \lambda(\varepsilon) \eta^{\varepsilon,h^\varepsilon}(s)) dB^H(s)\right|^{2} \nonumber\\
            &\leq C\frac{1}{\lambda^2(\varepsilon)} \sup_{0 \leq t \leq T} \left[ \int_0^t \left| \mathbb{E}(e^{\frac{s-t}{\mu}} \sigma(\bar{X}^{0}(s) + \varepsilon^H \lambda(\varepsilon) \eta^{\varepsilon,h^\varepsilon}(s)) \right|^{2} ds \right. \nonumber\\
            &\quad \left. +\mathbb{E} \int_0^t \left( \int_0^t |D_{s} (e^{\frac{r-t}{\mu}} \sigma(\bar{X}^{0}(r) + \varepsilon^H \lambda(\varepsilon) \eta^{\varepsilon,h^\varepsilon}(r))|^{\frac{1}{H}}ds \right)^{2H} dr \right]  \nonumber\\
            &\leq C\frac{1}{\lambda^2(\varepsilon)} \sup_{0 \leq t \leq T} \left[ \frac{\mu}{2} (1-e^{-\frac{2t}{\mu}}) K^2 + \int_0^t K^2 e^{\frac{2(r-t)}{\mu}} t^{2H} dr \right]  \nonumber\\
            &\leq C\frac{ K^2 }{\lambda^2(\varepsilon)}\frac{\mu}{2} \left(1 +  T^{2H}\right) = C_{K,T} \frac{\mu}{\lambda^2(\varepsilon)}.
        \end{align}
	Making use of assumption \textbf{(A1)}, \textbf{(A2)} and \eqref{eq:fractional_integral_estimate}, \eqref{eq:exponential_fractional_estimate}, we easily deduce that
	
	\begin{align*}
            \mathbb{E}\sup_{0\leq t\leq T}|{\eta^{\varepsilon,h^{\varepsilon}}}(t)|^{2} 
            &\leq C\int_{0}^{T}\mathbb{E}\sup_{0\leq r\leq s}|{\eta^{\varepsilon,h^{\varepsilon}}}(r)|^{2}ds \\
            &\quad + \left(\frac{\mu^{2}}{\varepsilon^{2H}\lambda^{2}(\varepsilon)}+1+\mu+\frac{1}{\lambda^{2}(\varepsilon)}+\frac{\mu}{\lambda^{2}(\varepsilon)}\right)C_{N,K,T}.
        \end{align*}
	
	Then the Gronwall's inequality yields
	
	\begin{align*}
            \mathbb{E}\sup_{0\leq t\leq T}|{\eta^{\varepsilon,h^{\varepsilon}}}(t)|^{2} 
            \leq \left(\frac{\mu^{2}}{\varepsilon^{2H}\lambda^{2}(\varepsilon)}+1+\mu+\frac{1}{\lambda^{2}(\varepsilon)}+\frac{\mu}{\lambda^{2}(\varepsilon)}\right)C_{N,K,T}e^{CT}.
        \end{align*}
	
	Since \({\varepsilon}^H\lambda(\varepsilon), \frac{1}{\lambda(\varepsilon)}, \frac{\mu}{{\varepsilon}^H}, \mu\to 0\) as \(\varepsilon\to 0\), we thus complete the proof.
\end{proof}

The verification of (ii) in Lemma \ref{mdp} is given in the next proposition.

\begin{prop}\label{mdpproba}
	Suppose that assumptions \textbf{(A1)}--\textbf{(A4)} hold and \(\lim_{\varepsilon\to 0}\frac{\mu}{{\varepsilon}^H}=0\). Let \(\{h^{\varepsilon}\}_{\varepsilon>0}\subset\mathscr{A}_{N}\) for some \(N<\infty\). Then for any \(\delta>0\), we have
	\begin{align*}
            \lim_{\varepsilon\to 0}\mathbb{P}\left\{d\left(\Upsilon^{\varepsilon}\left(\frac{1}{\lambda(\varepsilon)}B^H(\cdot)+\int_{0}^{\cdot}h^\varepsilon(s)ds\right),\Upsilon^{0}\left(\int_{0}^{\cdot}h^\varepsilon(s)ds\right)\right)>\delta\right\}= 0.
        \end{align*}
\end{prop}

\begin{proof}
	Recall that \({\eta^{\varepsilon,h^{\varepsilon}}}(t)=\Upsilon^{\varepsilon}\left(\frac{1}{\lambda(\varepsilon)}B^H(t)+\int_{0}^{t}h^\varepsilon(s)ds\right)\) and \(\tilde{X}^{h^\varepsilon}(t)=\Upsilon^{0}\left(\int_{0}^{t}h^\varepsilon(s)ds\right)\). Note that \({\eta^{\varepsilon,h^{\varepsilon}}}(t)\) has the representation (\ref{eta}) and \(\tilde{X}^{h^\varepsilon}(t)\) satisfies the following equation
	\begin{align*}
            \tilde{X}^{h^\varepsilon}(t) = \int_{0}^{t}\nabla b(\bar{X}^0(s))\tilde{X}^{h^\varepsilon}(s)ds 
            + \int_{0}^{t}\sigma(\bar{X}^0(s)) h(s)ds.
        \end{align*}
	
	Then \({\eta^{\varepsilon,h^{\varepsilon}}}(t)-\tilde{X}^{h^\varepsilon}(t)\) can be decomposed as the next three parts:
	\begin{align*}
            & {\eta^{\varepsilon,h^{\varepsilon}}}(t)-\tilde{X}^{h^\varepsilon}(t) \\
            &= \biggl[\frac{1}{\varepsilon^H\lambda(\varepsilon)}\int_0^{t} b(\bar{X}^{0}(s) + \varepsilon^H \lambda(\varepsilon) \eta^{\varepsilon,h^\varepsilon}(s))-b(\bar{X}^{0}(s))ds \\
            &\quad\quad\quad\quad\quad - \int_0^{t}\nabla b(\bar{X}^0(s))\tilde{X}^{h^\varepsilon}(s)ds\biggr] \\
            &\quad + \left[\int_0^{t}\left[\sigma(\bar{X}^{0}(s) + \varepsilon^H \lambda(\varepsilon) \eta^{\varepsilon,h^\varepsilon}(s))-\sigma(\bar{X}^{0}(s))\right]{h}^{\varepsilon}(s)ds\right] \\
            &\quad + \biggl[\frac{\mu}{\varepsilon^H\lambda(\varepsilon)}y_{0}(1-e^{-\frac{t}{\mu}})-\frac{1}{\varepsilon^H\lambda(\varepsilon)}e^{-\frac{t}{\mu}}\int_0^{t}e^{\frac{s}{\mu}}b(\bar{X}^{0}(s) + \varepsilon^H \lambda(\varepsilon) \eta^{\varepsilon,h^\varepsilon}(s))ds \\
            &\quad\quad\quad\quad\quad - e^{-\frac{t}{\mu}}\int_0^{t}e^{\frac{s}{\mu}}\sigma(\bar{X}^{0}(s) + \varepsilon^H \lambda(\varepsilon) \eta^{\varepsilon,h^\varepsilon}(s)){h}^{\varepsilon}(s)ds \\
            &\quad\quad\quad\quad\quad + \frac{1}{\lambda(\varepsilon)}\int_0^{t}\sigma(\bar{X}^{0}(s) + \varepsilon^H \lambda(\varepsilon) \eta^{\varepsilon,h^\varepsilon}(s))dB^H(s) \\
            &\quad\quad\quad\quad\quad - \frac{1}{\lambda(\varepsilon)}e^{-\frac{t}{\mu}}\int_0^{t}e^{\frac{s}{\mu}}\sigma(\bar{X}^{0}(s) +\varepsilon^H \lambda(\varepsilon) \eta^{\varepsilon,h^\varepsilon}(s))dB^H(s)\biggr] \\
            &= \sum_{i=1}^{3}I_{i}.
        \end{align*}
	
	For the term \(I_{1}\), by the mean value theorem with \(r \in [0,1]\), we have
	\begin{align*}
            I_{1} &= \left[\frac{1}{\varepsilon^H\lambda(\varepsilon)}\int_{0}^{t}\left[b(\bar{X}^{0}(s) + \varepsilon^H \lambda(\varepsilon) \eta^{\varepsilon,h^\varepsilon}(s))-b(\bar{X}^{0}(s))\right]ds\right.  \left. -\int_{0}^{t}\nabla b(\bar{X}^0(s))\tilde{X}^{h^\varepsilon}(s)ds\right] \\
            &= \int_{0}^{t}\int_{0}^{1}\nabla b(\bar{X}^{0}(s) + r\varepsilon^H \lambda(\varepsilon) \eta^{\varepsilon,h^\varepsilon}(s))\eta^{\varepsilon,h^\varepsilon}(s)drds  -\int_{0}^{t}\nabla b(\bar{X}^0(s))\tilde{X}^{h^\varepsilon}(s)ds.
        \end{align*}

	By \textbf{(A4)} we have
    
	\begin{align*}
            |I_{1}| &= \int_0^{t} \left| \int_0^{1} \nabla b(\bar{X}^{0}(s) + r\varepsilon^H \lambda(\varepsilon) \eta^{\varepsilon,h^\varepsilon}(s)) \eta^{\varepsilon,h^\varepsilon}(s) dr \right.  \left. - \nabla b(\bar{X}^0(s)) \tilde{X}^{h^\varepsilon}(s) \right| ds \\
            &\leq \int_0^{t} \left| \int_0^{1} \nabla b(\bar{X}^{0}(s) + r\varepsilon^H \lambda(\varepsilon) \eta^{\varepsilon,h^\varepsilon}(s)) \eta^{\varepsilon,h^\varepsilon}(s) \right. \\
            &\quad \left. - \nabla b(\bar{X}^{0}(s) + r\varepsilon^H \lambda(\varepsilon) \eta^{\varepsilon,h^\varepsilon}(s)) \tilde{X}^{h^\varepsilon}(s) dr \right| ds \\
            &\quad + \int_0^{t} \left| \int_0^{1} \nabla b(\bar{X}^{0}(s) + r\varepsilon^H \lambda(\varepsilon) \eta^{\varepsilon,h^\varepsilon}(s)) \tilde{X}^{h^\varepsilon}(s) dr \right.  \left. - \nabla b(\bar{X}^0(s)) \tilde{X}^{h^\varepsilon}(s) \right| ds \\
            &\leq C \int_0^{t} |\eta^{\varepsilon,h^\varepsilon}(s) - \tilde{X}^{h^\varepsilon}(s)| ds  + C \int_0^{t} \varepsilon^H \lambda(\varepsilon) |\eta^{\varepsilon,h^\varepsilon}(s)| |\tilde{X}^{h^\varepsilon}(s)| ds.
        \end{align*}
	
	Then combining (\ref{bound}) and (\ref{etabound}), we deduce that
	\begin{align*}
            \mathbb{E}\sup_{0\leq t\leq T}|I_{1}|^{2} 
            \leq C\int_{0}^{T}\mathbb{E}\sup_{0\leq r\leq s}|\eta^{\varepsilon,h^\varepsilon}(r)-\tilde{X}^{h^\varepsilon}(r)|^{2}ds + \varepsilon^{2H}\lambda^{2}(\varepsilon)C_{N,K,T}.
        \end{align*}
	
	Applying the Cauchy-Schwarz inequality and recalling \(h^{\varepsilon}\in\mathscr{A}_{N}\), we obtain from assumption \textbf{(A2)} and (\ref{etabound}) that
	\begin{align*}
            \mathbb{E}\sup_{0\leq t\leq T}|I_{2}|^{2} 
            &\leq C\mathbb{E}\int_{0}^{T}|\sigma(\bar{X}^{0}(s) + \varepsilon^H \lambda(\varepsilon) \eta^{\varepsilon,h^\varepsilon}(s))-\sigma(\bar{X}^{0}(s))|^{2}ds \int_{0}^{T}|h^\varepsilon(s)|^{2}ds \\
            &\leq CN\int_{0}^{T}|\varepsilon^H\lambda(\varepsilon)\eta^{\varepsilon,h^\varepsilon}(s)|^{2}ds \\
            &\leq \varepsilon^{2H}\lambda^{2}(\varepsilon)C_{N,K,T}.
        \end{align*}
	
	By Lemma \ref{mdpest}, we deduce that
	\begin{align*}
            \mathbb{E}\sup_{0\leq t\leq T}|I_{3}|^{2} 
            \leq \left(\frac{\mu^{2}}{\varepsilon^{2H}\lambda^{2}(\varepsilon)}+\mu+\frac{1}{\lambda^{2}(\varepsilon)}+\frac{\mu}{\lambda^{2}(\varepsilon)}\right)C_{N,K,T}.
        \end{align*}
	
	Thus, we arrived at
	\begin{align*}
            \mathbb{E}\sup_{0\leq t\leq T}|{\eta^{\varepsilon,h^{\varepsilon}}}(t)-\tilde{X}^{h^\varepsilon}(t)|^{2} 
            &\leq C\int_{0}^{T}\mathbb{E}\sup_{0\leq r\leq s}|\eta^{\varepsilon,h^\varepsilon}(r)-\tilde{X}^{h^\varepsilon}(r)|^{2}ds \\
            &\quad + \left(\varepsilon^{2H}\lambda^{2}(\varepsilon)+\frac{\mu^{2}}{\varepsilon^{2H}\lambda^{2}(\varepsilon)}+\mu+\frac{1}{\lambda^{2}(\varepsilon)}+\frac{\mu}{\lambda^{2}(\varepsilon)}\right)C_{N,K,T}.
        \end{align*}
	
	Then the Gronwall's inequality yields
        \begin{align*}
            \mathbb{E}\sup_{0\leq t\leq T}|{\eta^{\varepsilon,h^{\varepsilon}}}(t)-\tilde{X}^{h^\varepsilon}(t)|^{2} 
            &\leq \left(\varepsilon^{2H}\lambda^{2}(\varepsilon)+\frac{\mu^{2}}{\varepsilon^{2H}\lambda^{2}(\varepsilon)}+\mu\right.\\
            &\quad\left.+\frac{1}{\lambda^{2}(\varepsilon)}+\frac{\mu}{\lambda^{2}(\varepsilon)}\right) C_{N,K,T}e^{CT}.
        \end{align*}
	
	Since \({\varepsilon}^H\lambda(\varepsilon), \frac{1}{\lambda(\varepsilon)}, \frac{\mu}{{\varepsilon}^H}, \mu\to 0\) as \(\varepsilon\to 0\), applying Chebyshev's inequality completes the proof.
\end{proof}

\textbf{Acknowledgement} ~I would like to sincerely thank my supervisor Associate Professor Qian Yu, who has led the way to this work. 

\textbf{Data Availability Statements} ~The data that support the findings of this study are available from the corresponding author upon reasonable request.

\textbf{Declaration of interests} ~The authors declare that they have no known competing financial interests or personal relationships that
could have appeared to influence the work reported in this paper.	

	\bibliographystyle{plain}
	\addcontentsline{toc}{section}{References}
	\bibliography{zs.bib}

\end{document}